\documentclass[12pt]{amsart}
\usepackage{amscd,verbatim}
\usepackage{amssymb,amsmath,array}
\usepackage[all]{xy}
\usepackage{url}
\usepackage{enumerate}
\usepackage{bm}
\usepackage{mathrsfs}
\usepackage{color}
\usepackage{multirow}
\newcommand{\ul}{\underline}
\newcommand{\ol}{\overline}
\newcommand{\wt}{\widetilde}

\newcommand{\e}{\epsilon}

\newcommand{\C}{\mathbb{C}}

\newcommand{\G}{\mathbb{G}}

\renewcommand{\P}{\mathbb{P}}
\newcommand{\Q}{\mathbb{Q}}
\newcommand{\R}{\mathbb{R}}

\newcommand{\Z}{\mathbb{Z}}

\newcommand{\sM}{\mathscr{M}}
\newcommand{\sO}{\mathscr{O}}

\newcommand{\sS}{\mathscr{S}}

\newcommand{\sX}{\mathscr{X}}

\newcommand{\Aut}{\operatorname{Aut}}

\newcommand{\Tor}{{\operatorname{tors}}}

\renewcommand{\Im}{\operatorname{Im}}

\newcommand{\ch}{\operatorname{char}}

\newcommand{\id}{{\operatorname{id}}}

\newcommand{\cube}{\operatorname{\square}}

\newcommand{\is}{{\operatorname{is}}}

\def\H{\mathbb H}
\def\M#1#2#3#4{\begin{pmatrix}#1&#2\\#3&#4\end{pmatrix}}
\def\gen#1{\langle #1\rangle}

\renewcommand{\div}{\operatorname{div}}
\newcommand{\Div}{\operatorname{Div}}
\newcommand{\Pic}{\operatorname{Pic}}

\newcommand{\Spec}{\operatorname{Spec}}

\newcommand{\CH}{{\operatorname{CH}}}

\newcommand{\SL}{{\operatorname{SL}}}
\newcommand{\ord}{{\operatorname{ord}}}

\newcommand{\lc}{{\operatorname{lc}}}
\newcommand{\LC}{{\operatorname{LC}}}
\newcommand{\ram}{{\operatorname{ram}}}

\newcommand{\WP}{{\operatorname{Weil}}}
\def\SM#1#2#3#4{\left(\begin{smallmatrix}#1&#2\\#3&#4\end{smallmatrix}\right)}

\theoremstyle{plain}
\newtheorem{theorem}{Theorem}[section]
\newtheorem{proposition}[theorem]{Proposition}
\newtheorem{lemma}[theorem]{Lemma}
\newtheorem{corollary}[theorem]{Corollary}

\theoremstyle{definition}
\newtheorem{definition}[theorem]{Definition}
\newtheorem{example}[theorem]{Example}
\newtheorem{remark}[theorem]{Remark}

\numberwithin{equation}{section}
\begin{document}

\title[The intrinsic subgroup of an elliptic curve]
{The intrinsic subgroup of an elliptic curve and Mazur's torsion theorem}

%\title{An intrinsic subgroup of 
%\\ torsion rational points on an elliptic curve}

%\title[Refinement of Mazur's torsion theorem]
%{Refinement of Mazur's torsion theorem \\ via the intrinsic subgroups}

\date{\today}

\author{Takao Yamazaki}
\address{Takao Yamazaki, Department of Mathematics,
Chuo University, 1-13-27 Kasuga,
Bunkyo-ku, Tokyo 112-8551, Japan}
\email{ytakao@math.chuo-u.ac.jp}

\author{Yifan Yang}
\address{Yifan Yang, Department of Mathematics, National Taiwan University and National Center for Theoretical Science, Taipei 10617, Taiwan}
\email{\textnormal{yangyifan@ntu.edu.tw}}

\author{Hwajong Yoo}
\address{Hwajong Yoo, College of Liberal Studies and Research Institute of Mathematics, Seoul National University, Seoul 08826, South Korea}
\email{\textnormal{hwajong@snu.ac.kr}}                                                                  
\author{Myungjun Yu}
\address{Myungjun Yu, Department of Mathematics, Yonsei University, Seoul 03722, South Korea}
\email{\textnormal{mjyu@yonsei.ac.kr}}      

\begin{abstract}
We define and study a biadditive symmetric 
(not necessarily perfect) pairing on 
the torsion part $\mathrm{Pic}(X)_{\mathrm{tors}}$ of the Picard group of 
a smooth projective curve $X$ over a field $k$
with values in $k^\times \otimes \mathbb{Q}/\mathbb{Z}$.
We call its kernel the \textit{intrinsic subgroup} of $X$.
It turns out that
some information on the reduction type of $X$
can be read off from the intrinsic subgroup.
Mazur's torsion theorem says that 
there are exactly 15 isomorphism classes of abelian groups
that appear as the rational torsion points of 
an elliptic curve $X$ over $\mathbb{Q}$
(identified with $\mathrm{Pic}(X)_{\mathrm{Tor}}$).
We refine this result by determining 
which subgroups of those 15 groups appear as the intrinsic subgroups.
\end{abstract}

   % 11F03
%\thanks{The first author is supported by JSPS KAKENHI Grant
%  (25K06961). The second author is supported by Grant
%  113-2115-M-002-003-MY3 of the National Science and Technology
%  Council of the Republic of China (Taiwan). The third author (H. Y.) is supported by National Research Foundation of Korea(NRF) grant funded by the Korea government(MSIT) (No. RS-2023-00239918  and No. 2020R1A5A1016126).
%  The fourth author is supported by the National Research Foundation of Korea (NRF) grant funded by the
%Korea government (MSIT) (RS-2025-23525445).}
\maketitle

\section{Introduction}
Let $X$ be a geometrically irreducible smooth projective curve
over a field $k$. % of characteristic zero.
In \S \ref{sect:pairing},
we shall construct a biadditive symmetric pairing
\begin{equation}\label{eq:pairing-intro}
\langle \cdot , \cdot \rangle :
\Pic(X)_\Tor \times \Pic(X)_\Tor \to k^\times \otimes \Q/\Z,
\end{equation}
where $\Pic(X)_\Tor$ is the 
torsion part of the Picard group $\Pic(X)$ of $X$.
We then define the \emph{intrinsic subgroup}
of $\Pic(X)_\Tor$ by
\[
\Pic(X)_\Tor^\is 
:= \{ a \in \Pic(X)_\Tor \mid \langle a, b \rangle = 0
\text{ for all } b \in \Pic(X)_\Tor \}.
\]

We remark that \eqref{eq:pairing-intro} is not entirely new.
Indeed, its finite coefficient analogue 
(with values in $k^\times/(k^\times)^m$)  %, and for $k$ a finite field)
is considered by Frey and R\"uck in \cite{FreyRuck}
(see Remark \ref{rem:FreyRuck}).
The first and second authors encountered with
a disguised version of \eqref{eq:pairing-intro} in \cite{YY}
(see Remark \ref{sect:genJac}).
Obviously, \eqref{eq:pairing-intro} is trivial if $k^\times \otimes \Q/\Z=0$,
e.g. when $k$ is finite, algebraically closed, or $k=\R$.
On the other hand,
$\Pic(X)_\Tor^\is$ appears to be a non-trivial new invariant
if $k$ is a global field or a $p$-adic field.

We will mainly consider the case where $X=E$ is an elliptic curve
so that we may identify $E(k)_\Tor= \Pic(E)_\Tor$
by $P \mapsto [P-O]$, where $O \in E(k)$ is the identity element.
Thus we write 
$\gen{P,Q}:=\gen{[P-O], [Q-O]}$ for $P, Q \in E(k)_\Tor$
and 
%$E(k)_\Tor:=\Pic(k)_\Tor$,
$E(k)_\Tor^\is:=\Pic(E)_\Tor^\is$.
Our main result is a classification of possible structures of
$(E(\Q)_\Tor,E(\Q)_\Tor^\is)$ for an elliptic curve $E$ over $\Q$.
A celebrated theorem of Mazur \cite[Theorem (8)]{Mazur} states
that $E(\Q)_\Tor$ is either cyclic of order
$1, \ldots,10$, or $12$, or is isomorphic to $(\Z/2N\Z) \times (\Z/2\Z)$,
$N=1, 2,3,4$,
and all cases are realized by infinitely many (mutually non-isomorphic) elliptic curves.
The classification of
$E(\Q)_\Tor^\is$ is given in the following theorem,
whose proof will occupy \S \ref{sect:ref-mazur}--\ref{sect:ref-mazur2}.

\begin{theorem} \label{theorem: intrinsic torsion}
  Let $E$ be an elliptic curve over $\Q$.
Then there is a pair $(A, B)$ of an abelian group $A$
and its subgroup $B$ in the following table,
together with 
an isomorphism $\alpha : E(\Q)_\Tor \cong A$, 
such that $\alpha(E(\Q)_\Tor^\is)=B$.
  $$ \extrarowheight3pt
  \begin{array}{ll} \hline\hline
    A & \text{order of}\ B \\ \hline
    0 & 1 \\
    \Z/2\Z & 1,2\\
    \Z/3\Z & 1,3\\
    \Z/4\Z & 1,2,4 \\
    \Z/5\Z & 1,5 \\
    \Z/6\Z & 1,2,3 \\
    \Z/7\Z & 1 \\
    \Z/8\Z & 1,2 \\
    \Z/9\Z & 1 \\
    \Z/10\Z & 1\\
    \Z/12\Z & 1 \\
    \Z/2\Z\times\Z/2\Z & 1,2\\
    \Z/4\Z\times\Z/2\Z & 1,2 \ \\
    \Z/6\Z\times\Z/2\Z & 1 \\
    \Z/8\Z\times\Z/2\Z & 1 \\ \hline\hline
  \end{array}
  $$
%  Then $E(\Q)_\Tor$ and
%  $E(\Q)_\Tor^\is$ are isomorphic to one of the pairs in the following
%  table.
%  $$ \extrarowheight3pt
%  \begin{array}{ll} \hline\hline
%    E(\Q)_\Tor & E(\Q)_\Tor^\is \\ \hline
%    0 & 0 \\
%    \Z/2\Z & \Z/M\Z,~M=1,2\\
%    \Z/3\Z & \Z/M\Z,~M=1,3\\
%    \Z/4\Z & \Z/M\Z,~M=1,2,4 \\
%    \Z/5\Z & \Z/M\Z,~M=1,5 \\
%    \Z/6\Z & \Z/M\Z,~M=1,2,3 \\
%    \Z/7\Z & 0 \\
%    \Z/8\Z & \Z/M\Z,~M=1,2 \\
%    \Z/9\Z & 0 \\
%    \Z/10\Z & 0\\
%    \Z/12\Z & 0 \\
%    \Z/2\Z\times\Z/2\Z & \Z/M\Z,~M=1,2\\
%    \Z/4\Z\times\Z/2\Z & \Z/M\Z,~M=1,2\\
%    \Z/6\Z\times\Z/2\Z & 0 \\
%    \Z/8\Z\times\Z/2\Z & 0 \\ \hline\hline
%  \end{array}
%  $$
(For $A=\Z/4\Z \times \Z/2\Z$,
there are two $\Aut(A)$-conjugacy classes of order two subgroups,
both of which can be taken as $B$.)
  Moreover, for each possible pair $(A, B)$ as above,
there are infinitely many (mutually non-isomorphic) elliptic curves $E$ over $\Q$ 
for which there is an isomorphism $\alpha : E(\Q)_\Tor \cong A$
such that $\alpha(E(\Q)_\Tor^\is)=B$.
\end{theorem}

%\begin{remark}
%For $A=\Z/4\Z \times \Z/2\Z$,
%there are two $\Aut(A)$-conjugacy classes of order two subgroups,
%both of which can be taken as $B$.
%\end{remark}

We briefly explain the outline of the proof.
Our approach is completely explicit. 
As usual, we let $X_1(N)$ denote the modular curve associated to the
congruence subgroup $\Gamma_1(N)$. It possesses a model over $\Q$ on
which the cusp $0$ is $\Q$-rational. For a subfield $k$ of $\C$, 
%we let $X_1(N)(k)$ denote the set of $k$-rational points on this model. Then 
the non-cuspidal points in $X_1(N)(k)$ parameterize isomorphism
classes of pairs $(E, P)$ of an elliptic curve $E$ over $k$ 
and a $k$-rational point $P$ of order $N$. 
More precisely, if the coordinates of a point 
$\tau\in\H$ on this model of $X_1(N)$ is $k$-rational, 
then the isomorphism class of $(\C/(\Z\tau+\Z),1/N)$ contains a pair $(E,P)$
of an elliptic curve $E$ over $k$ and a $k$-rational $N$-torsion point $P$.
A famous theorem of Mazur \cite[Theorem (7)]{Mazur} says that 
$X_1(N)(\Q)$ has a non-cuspidal point precisely
when $N=1,\ldots,10$, or $N=12$. 

For a pair $(N, M)$ of $N \in \{4, \dots, 10, 12\}$ and a positive divisor $M$ of $N$,
we will construct a smooth projective curve $X_1 (N)^\pm_M$ over $\Q$ equipped with
a finite morphism $X_1(N)_M^\pm \to X_1(N)$
that enjoys the following property:
when $t \in X_1(N)(\Q)$ corresponds to a pair $(E, P)$ as above (with $k=\Q$),
one has 
\[
\langle P, (N/M)P \rangle = 0 \Leftrightarrow
t \in \Im(X_1(N)^\pm_M(\Q) \to X_1(N)(\Q)).
\]
It is worth  mentioning that
$X_1(N)^+_M$ is isomorphic to the modular curve
$X_1(MN, N)$ associated to $\Gamma_1(MN, N):=\Gamma_0(MN) \cap \Gamma_1(N)$
(Theorem \ref{proposition: pairing 1-mod})
and $X_1(N)^-_M$ is its twist.

Similarly, for a positive even
integer $N$, we let $X_1^0(N,2)$ denote the modular curve associated
to the congruence subgroup
\begin{equation}\label{eq:G10-N2}
\Gamma_1^0(N,2):=\Gamma_1(N)\cap\Gamma^0(2)
=\left\{\M abcd\in\SL(2,\Z):a,d\equiv 1\bmod N,~~2|b,~N|c\right\}.
\end{equation}
The modular curve $X_1^0(N,2)$ parameterizes isomorphism classes of
triples $(E, P, Q)$ of an elliptic curve $E$
and $k$-rational points $P, Q$ 
such that 
%$P$ is of order $N$, $Q$ is of order $2$, 
$P$ and $Q$ are of order $N$ and $2$ respectively,
and $Q\notin\gen P$.
For $N \in \{4, 6, 8 \}$
and $\ul{M}=(M_1, M_2, M_3)$ with $M_1|N, M_2, M_3 \in \{ 1, 2\}$,
we will construct a smooth projective curve  
$X_1^0(N, 2)_{\ul{M}}^\pm$ over $\Q$
equipped with a finite morphism to $X_1^0(N, 2)$,
by which one can interpret the conditions
\[ \langle P, \frac{N}{M_1}P \rangle 
= \langle Q, \frac{2}{M_2}Q \rangle 
= \langle P, \frac{2}{M_3}Q \rangle = 0
\]
for a triple $(E, P, Q)$ as above corresponding to a point of $X_1^0(N, 2)(\Q)$.
Therefore Theorem \ref{theorem: intrinsic torsion} is reduced to a study of the $\Q$-rational points of 
$X_1(N)_M^\pm$ and $X_1^0(N, 2)_{\ul{M}}^\pm$.
As our construction of these curves are explicit,
this can be done by a (more or less) direct computation.

In the last section \S \ref{sect:high-dim},
we generalize the pairing \eqref{eq:pairing-intro}
and the intrinsic subgroup 
to higher dimensional varieties.
This new construction enables us to prove that,
if $X$ has good reduction with respect to a discrete valuation of $k$,
there is a strong restriction on the values of the pairing 
(see Proposition \ref{prop:two-def-same}).
Finally, as a sample for the case of bad reduction,
we compute the intrinsic subgroup of 
Tate elliptic curves (see Proposition \ref{prop:TateCv}).
As an application,
we show that the intrinsic subgroup of an elliptic curve over a number field
imposes some constraint on its reduction type
(see Corollary \ref{cor:kodairatype}).

\subsection*{Acknowledgement}
It is Kenneth Ribet who formulated \eqref{eq:pairing-intro} as a pairing 
and asked if it is symmetric.
The authors express deep gratitude to him for asking this question,
which was a starting point of our work.
We also thank the referee for his or her
careful reading and pointing out errors in the earlier version.

The first author (T. Y.) is supported by JSPS KAKENHI Grant (25K06961). 
The second author (Y. Y.) is supported by 
Grant 113-2115-M-002-003-MY3 of the National Science and Technology
  Council of the Republic of China (Taiwan). 
The third author (H. Y.) is supported by National Research Foundation of Korea(NRF) grant 
funded by the Korea government(MSIT) (No. RS-2023-00239918  and No. 2020R1A5A1016126).
The fourth author (M. Y.) is supported by the National Research Foundation of Korea (NRF) grant 
funded by the
Korea government (MSIT) (RS-2025-23525445).

\subsection*{Notation and convention}
For an abelian group $A$, 
we write $A[n]$ and $A/n$ for the kernel and cokernel
of $A \to A, \ a \mapsto na$ for each $n \in \Z$,
and put $A_\Tor := \cup_{n > 0} A[n]$.
For a field $k$,
we set $\mu_n(k):=(k^\times)[n]$ and $\mu(k):=(k^\times)_\Tor$.

The identity element of an elliptic curve is denoted by $O$.

\section{Pairing on the Picard group and the intrinsic subgroup}
\label{sect:pairing}
Throughout this section, 
we let $X$ be a geometrically irreducible smooth projective curve
over a field $k$. % of characteristic zero.
The goal of this section is to construct a
biadditive pairing
\begin{equation}\label{eq:pairing}
\langle \cdot , \cdot \rangle :
\Pic(X)_\Tor \times \Pic^0(X) \to k^\times \otimes \Q/\Z
\end{equation}
which restricts to a symmetric pairing \eqref{eq:pairing-intro}.
Here  $\Pic^0(X) := \ker(\deg : \Pic(X) \to \Z)$
(thus $\Pic(X)_\Tor \subset \Pic^0(X)$).
This will be generalized to higher dimensional varieties 
in Section \ref{sect:high-dim} below,
but our construction in this section
is more elementary and sufficient for most of our discussion.

\subsection{Construction of the pairing}

For a non-zero rational function $f \in k(X)^\times$,
a closed point $Q \in X$,
and $n \in \Z_{>0}$ such that $n|\ord_Q(f)$,
we define the \emph{leading coefficient} by
\[
\lc_Q(f, n):=
\left(\frac{f}{\pi^{\ord_{Q}(f)}} (Q) \right) \otimes \frac{1}{n}
\in k(Q)^\times \otimes \Q/\Z,
\]
where $\pi \in k(X)^\times$ is a uniformizer at $Q$.
This is independent of the choice of $\pi$:
replacing $\pi$ by $u \pi$
for some $u \in k(X)^\times$ with $\ord_Q(u)=0$
changes $\lc_Q(f, n)$ by
$u(Q)^{-\ord_Q(f)} \otimes (1/n)$
which vanishes by the assumption $n|\ord_Q(f)$.
Note also that for any $m>0$ we have
\begin{equation}\label{eq:lc-power-m}
\lc_Q(f, n)=\lc_Q(f^m, mn).
\end{equation}

We define a subgroup $\Div_t(X)$
of the group of divisors $\Div(X)$ on $X$ by 
\[
\Div_t(X):=\{ D \in \Div(X) \mid [D] \in \Pic(X)_\Tor \},
\]
where we write $[D] \in \Pic(X)$ for the class of $D \in \Div(X)$.
It contains all rationally trivial divisors,
and all its members have degree zero:
\[ \{ \div(f) \mid f \in k(X)^\times \}
\subset \Div_t(X) \subset \Div^0(X) := \ker(\deg : \Div(X) \to \Z).
\]
Given $D \in \Div_t(X)$ and 
$E = \sum_j e_j Q_j \in \Div^0(X)$,
we define
\begin{equation}\label{eq:pairing1}
\langle D, E\rangle :=
\prod_j 
N_{k(Q_j)/k}(\lc_{Q_j}(f, n))^{e_j} 
\in k^\times \otimes \Q/\Z,
\end{equation}
where
$n \in \Z_{>0}$ is such that $n[D]=0$,
$f \in k(X)^\times$ is such that $\div(f)=nD$
(so that $n | \ord_P(f)$ for any $P$),
%and $\pi_{Q_j} \in k(X)^\times$ is a uniformizer at $Q_j$.
and $N_{k(Q_j)/k} : k(Q_j)^\times \otimes \Q/\Z \to k^\times \otimes \Q/\Z$
is induced by the norm map attached to the finite extension $k(Q_j)/k$.

\begin{proposition}
For $D \in \Div_t(X)$ and $E \in \Div^0(X)$, 
the element \eqref{eq:pairing1} depends 
only on the classes of $D, E$ in $\Pic(X)$.
Consequently we obtain the induced biadditive pairing
\eqref{eq:pairing} (denoted by the same symbol).
%\[
%\langle \cdot, \cdot \rangle : \Pic(X)_\Tor \times \Pic^0(X) 
%\to k^\times \otimes \Q/\Z.
%\]
Moreover, its restriction to $\Pic(X)_\Tor \times \Pic(X)_\Tor$
is symmetric.
\end{proposition}

\begin{proof}
We first show that \eqref{eq:pairing1} is independent of the choices
of $f$ and $n$.
\begin{itemize}
\item 
Replacing $f$ by $cf$ for some $c \in k^\times$
does not change $\langle D, E\rangle$ 
since
$\prod_j N_{k(Q_j)/k}(c)^{e_j}
=\prod_j c^{[k(Q_j):k] e_j} =c^{\deg(E)}=1$ by $\deg(E)=0$.
\item 
For any $m \in \Z_{>0}$,
replacing $n$ by $mn$ 
does not change $\langle D, E\rangle$ 
since we have $\div(f^m)=mnD$
and \eqref{eq:lc-power-m}.
\end{itemize}
In particular, 
if $D=\div(f)$ for $f \in k(X)^\times$,
we may take $n=1$ to conclude
$\langle D, E\rangle=\langle \div(f), E \rangle=0$.
Thus \eqref{eq:pairing1} factors through
%(still written by $\langle \cdot , \cdot \rangle$)
$\langle \cdot , \cdot \rangle :
\Pic(X)_\Tor \times \Div^0(X) \to k^\times \otimes \Q/\Z.$

We now assume that $E$ is also in $\Div_t(X)$
and that $|D| \cap |E|=\emptyset$. 
We claim
\begin{equation}\label{eq:symm}
\langle D, E \rangle = \langle E, D \rangle.
\end{equation}
To show this, 
we suppose $n[D]=n[E]=0$ by replacing $n$ if necessary,
and write $D=\sum_i d_i P_i$
and $nE=\div(g)$ with $g \in k(X)^\times$.
Then we have $\lc_Q(f, n)=f(Q) \otimes (1/n)$ 
and $\lc_P(g, n)=g(P) \otimes (1/n)$ 
for all $P \in |D|$ and $Q \in |E|$,
and hence
\[
\langle D, E\rangle =
\prod_j N_{k(Q_j)/k}(f(Q_j))^{e_j}\otimes \frac{1}{n},
\qquad
\langle E, D\rangle =
\prod_i N_{k(P_i)/k}(g(P_i))^{d_i} \otimes \frac{1}{n}.
\]
Recall that the Weil pairing $\WP_n([D], [E])$
is given by
\[
\WP_n([D], [E])
=\frac{\prod_j N_{k(Q_j)/k}(f(Q_j))^{e_j}}{\prod_i N_{k(P_i)/k}(g(P_i))^{d_i}}
\]
which takes its value in  $\mu_n(k)$.
(This is a consequence of the Weil reciprocity
and valid in any characteristic,
although $\WP_n$ fails to be perfect if $p=\ch(k)>0$ and $p|n$.)
Hence it vanishes after taking $
- \otimes (1/n)$ in $k^\times \otimes \Q/\Z$.
We have shown \eqref{eq:symm}.

Finally, suppose that 
$E=\div(g)$ for some $g \in k(X)^\times$.
We can always find $h \in k(X)^\times$
such that $|D + \div(h)| \cap |\div(g)|=\emptyset$.
From what we have shown it follows that
\[ 
\langle D, \div(g) \rangle
= \langle D + \div(h), \div(g) \rangle
= \langle \div(g), D + \div(h) \rangle
= 0.
\]
This completes the proof of the proposition.
%We conclude that \eqref{eq:pairing1} for 
%$D \in \Div_t(X)$ and $E \in \Div^0(X)$
%yields a well-defined pairing \eqref{eq:pairing}
%which is symmetric on $\Pic(X)_\Tor \times \Pic(X)_\Tor$.
\end{proof}

\begin{remark}\label{rem:FreyRuck}
As is seen from the proof,
there is an analogous biadditive pairing
\[
\langle \cdot, \cdot \rangle_n : 
\Pic(X)[n] \times \Pic^0(X)/n \to k^\times/(k^\times)^n
\]
for each $n>0$,
which satisfies
$\langle d, e \rangle_n 
=
\WP_n(d, e)
\langle e, d \rangle_n
$for any $d, e \in \Pic(X)[n]$.
This pairing has been constructed by Frey and R\"uck in \cite{FreyRuck}
when $k$ is a finite field
and shown to be perfect if $|k| \equiv 1 \bmod n$.
(This is called the 
Lichtenbaum-Tate pairing in some literature,
see e.g. \cite[XI.9]{sil2}.)
\end{remark}

%\subsection{Relation with the generalized Jacobian}\label{sect:genJac}

\begin{remark}[Relation with the generalized Jacobian]\label{sect:genJac}
Let $P_0, \dots, P_r \in X(k)$ be distinct $k$-rational points on $X$
such that $P_i-P_0 \in \Div_t(X)$ for any $i$,
and put $D:=P_0+\cdots+P_r$.
Let $J$ (resp. $\wt{J}$) be the Jacobian of $X$
(resp. generalized Jacobian of $X$ with modulus $D$).
We have an exact sequence
\[ 0 \to T \to \wt{J} \to J \to 0, \]
where $T \cong \G_m^r$ is the (split) torus 
with character group $\Z[D]^0 := \oplus_{i=1}^r (P_i-P_0)\Z$.
It induces a diagram with exact row:
\begin{equation}
\label{eq:gen-jac}
\vcenter{
\xymatrix{
0 \ar[r] & 
T(k)_\Tor \ar[r] &
\wt{J}(k)_\Tor \ar[r] &
J(k)_\Tor \ar[r]  \ar[rd]_-\Phi &
H^1(k, T(\ol{k})_\Tor) \ar[d]^{\cong}
\\
& & &
\Z[D]^0 \ar[u]^{[ \cdot ]} &
(k^\times \otimes \Q/\Z)^r,
}}
\end{equation}
where 
$\Phi$ is defined by the commutativity.
It is shown in \cite[Lemma 2.3.1]{YY} that
$\Phi$ can be described in terms of the pairing \eqref{eq:lc-power-m} as
\begin{equation*}%\label{eq:genJac}
\Phi([P_i-P_0]) = (\langle P_i-P_0, P_j-P_0 \rangle)_{j=1}^r.
\end{equation*}
\end{remark}

%Under the identification
%\[
%\Hom(\Z[D]^0, k^\times \otimes \Q/\Z) \cong (k^\times \otimes \Q/\Z)^r,
%\qquad
%\alpha \mapsto (\alpha(P_j-P_0))_j,
%\]
%we find that $\Phi$ factors through a dotted arrow in the commutative diagram:
%\[
%\xymatrix{
%\Z[D]^0 \ar[d] \ar[r] 
%&
%\Hom(J(k)_\Tor, k^\times \otimes \Q/\Z) 
%\ar@{^{(}->}[d]
%\\
%J(k)_\Tor \ar[r]_-\Phi \ar@{.>}[ru]
%& \Hom(\Z[D]^0, k^\times \otimes \Q/\Z).
%}
%\]

\subsection{The intrinsic subgroup}

\begin{definition}
We define the \emph{intrinsic subgroup}
of $\Pic(X)_\Tor$ by
\begin{align*}
&\Pic(X)_\Tor^\is 
:= \{ a \in \Pic(X)_\Tor \mid \langle a, b \rangle = 0
\text{ for all } b \in \Pic(X)_\Tor \}.
\end{align*}
\end{definition}

%If $X$ has a $k$-rational point so that $\Pic^0(X)$ is
%canonically isomorphic to the group of $k$-valued points $J(k)$ 
%on the Jacobian variety $J$ of $X$,
%then we write 
%$J(k)_\Tor^\is:=\Pic(X)_\Tor^\is$.
If $X=E$ is an elliptic curve, we identify 
$E(k)=\Pic^0(X)$ and put $E(k)_\Tor^\is:=\Pic(X)_\Tor^\is$.

Here we exhibit a few examples in the case of modular curves.

\begin{example}
Let $p$ be a prime and consider the modular curve $X=X_0(p)$ over $k=\Q$.
Let $P_0$ and $P_1$ be the $0$ and $\infty$ cusps, respectively.
Mazur \cite[Theorem 1]{Mazur} 
shows that
$\Pic(X)_\Tor$ is a cyclic group of order $n:=(p-1)/(p-1, 12)$
generated by $[P_1-P_0]$.
It follows from \cite[Proposition 5.4.1]{YY} that
\[ \langle P_1 - P_0, P_1-P_0 \rangle 
= p \otimes \frac{12}{p-1} \in \Q^\times \otimes \Q/\Z. 
\]
This can be computed by taking
$f=(\eta(p\tau)/\eta(\tau))^{24/(p-1, 12)}$
so that $\div(f)=n(P_0-P_1)$.
In particular, we have $\Pic(X)_\Tor^\is=\{ 0 \}$.
\end{example}

\begin{example}
Let $p, q$ be distinct primes such that $p \equiv q \equiv 1 \bmod 12$,
and put
\[
a=\frac{(p-1)(q+1)}{24}, \quad
b=\frac{(p+1)(q-1)}{24}, \quad
c=\frac{(p-1)(q-1)}{24}.
\]
Let $X=X_0(pq)$ be the modular curve over $k=\Q$
and let $P_0, P_1, P_2, P_3$ be the cusps
of level $pq, 1, p, q$, respectively.
Ohta \cite{Ohta} shows $\Pic(X)_\Tor$ is generated by
$P_i-P_0 \ (i=1, 2, 3)$ up to $2$-primary torsion.
It follows from \cite[Section 7]{YY}
that the values of $\langle P_i-P_0, P_j-P_0 \rangle$
are given by $-\otimes (1/2ab)$ of the following table.
\[
\begin{array}{|c|ccc|} \hline
 & P_1-P_0 & P_2-P_0 & P_3-P_0
\\
\hline
P_1-P_0 & p^bq^a & q^a & p^b  \\
P_2-P_0 & q^a & q^a & 1  \\
P_3-P_0 & p^b & 1 & p^b  \\
\hline
\end{array}
\]
In particular,
$\Pic(X)_\Tor^\is$
is a cyclic group of order $c$ generated by 
%$P_0+P_1-P_2-P_3=(P_1-P_0)-(P_2-P_0)-(P_3-P_0)$.
$[P_0+P_1-P_2-P_3]$,
up to $2$-primary torsion.
\end{example}

\subsection{Functoriality}
Let $Y$ be another geometrically irreducible smooth projective curve over $k$
and $\phi : X \to Y$ a finite morphism over $k$.
We have
the push-forward $\phi_* : \Div(X) \to \Div(Y)$
and 
the pull-back $\phi^* : \Div(Y) \to \Div(X)$
of divisors along $\phi$.
Recall that they are given by
\begin{align*}
&\phi_*(\sum_i d_i P_i) = \sum_i d_i [k(P_i):k(\phi(P_i))] \phi(P_i),
\\
&\phi^*(\sum_j e_j Q_j) = \sum_j e_j 
\sum_{P \in \phi^{-1}(Q_j)} \ram_P(\phi) P,
\end{align*}
where $\ram_P(\phi)$ denotes the ramification index of $\phi$ at $P$.
We denote by
$\langle \cdot, \cdot \rangle_X$ and $\langle \cdot, \cdot \rangle_Y$
the pairing \eqref{eq:pairing} for $X$ and $Y$, respectively.

\begin{lemma}[Projection formula]
For any $D \in \Div_t(X)$ and $E \in \Div^0(Y)$, we have
\begin{equation}\label{eq:proj-f}
\langle \phi_*D, E \rangle_Y =
\langle D, \phi^*E \rangle_X.
\end{equation}
Consequently, we have 
\[
\phi^*(\Pic(Y)_\Tor^\is) \subset \Pic(X)_\Tor^\is,
\qquad
\phi_*(\Pic(X)_\Tor^\is) \subset \Pic(Y)_\Tor^\is.
\]
\end{lemma}
\begin{proof}
We may assume $|\phi_* D| \cap |E|=\emptyset$ and $|D| \cap |\phi^* E|=\emptyset$
by replacing $E$ by $E+\div(h)$ for a suitable $h \in k(Y)^\times$.
Take $n \in \Z_{>0}$ and $f \in k(X)^\times$ such that 
$n[D]=0$ and $\div(f)=nD$.
Since we have $\div(N_{k(X)/k(Y)}(f))=\phi_* \div(f)$,
the left and right hand sides of \eqref{eq:proj-f}  are rewritten as
\begin{align*}
\prod_j (N_{k(X)/k(Y)} f)(Q_j)^{e_j} \otimes \frac{1}{n},
\qquad
\prod_j 
\left(\prod_{P \in \phi^{-1}(Q_j)} f(P)^{\ram_P(\phi)}\right)^{e_j} 
\otimes \frac{1}{n}.
\end{align*}
Now \eqref{eq:proj-f} follows from the formula
$(N_{k(X)/k(Y)} f)(Q_j)
=\prod_{P \in \phi^{-1}(Q_j)} f(P)^{\ram_P(\phi)}$.
The last statement immediately follows from \eqref{eq:proj-f}. 
We are done.
\end{proof}

\begin{remark}
Let $k'/k$ be a field extension,
and $X'$ the base change of $X$ to $k'$.
Then we have
$\phi^* \langle D, E \rangle = \langle \phi^*D, \phi^*E \rangle'$,
%for any $D \in \Div_t(X)$ and $E \in \Div^0(X)$,
where 
$\phi^* : k^\times \otimes \Q/\Z \to {k'}^\times \otimes \Q/\Z$ 
and 
$\phi^* : \Pic(X) \to \Pic(X')$ are the induced maps,
and $\langle \cdot, \cdot \rangle'$ denotes 
\eqref{eq:pairing} for $X'$.
However, it can happen that $\phi^*(\Pic(X)_\Tor^\is) \not\subset \Pic(X')_\Tor^\is$
(when $\phi^*(\Pic(X)_\Tor) \subsetneq \Pic(X')_\Tor$).
\end{remark}

\section{Elliptic curves over $\Q$}\label{sect:ref-mazur}

Let $k$ be a field. Suppose that an elliptic curve
$E:y^2+a_1xy+a_3y=x^3+a_2x^2+a_4x+a_6$ over $k$ has a
$k$-rational point $P$ other than the point $O$ at infinity. By a suitable
change of variables, we may assume that $P=(0,0)$, so the equation of
the elliptic curve is of the form $y^2+a_1xy+a_3y=x^3+a_2x^2+a_4x$.
We can show that $(0,0)$ is a $2$-torsion if and only if $a_3=0$.
Thus, if $(0,0)$ is a $2$-torsion and $\operatorname{char}k\neq 2$,
then we may make a change of variables $(x,y)\mapsto(x,y-a_1x/2)$
and assume that $E$ is of the form
\begin{equation} \label{eq: E2}
E_{2,a,b}:y^2=x(x^2+ax+b), \qquad b(a^2-4b)\neq 0
\end{equation}
for some $a,b \in k$.
If $(0,0)$ is not a $2$-torsion (without assumptions on
$\operatorname{char}k$), then we have $a_3\neq0$. We set
$(x,y)\mapsto(x,y+a_4x/a_3)$ and obtain an equation of the form
$y^2+b_1xy+b_3y=x^3+b_2x^2$. We can show that $(0,0)$ is a $3$-torsion
if and only if $b_2=0$. Thus, if $(0,0)$ is a $3$-torsion, then we may
assume that the equation of the elliptic curve is of the form
\begin{equation} \label{eq: E3}
E_{3,a,b}:y^2+axy+by=x^3, \qquad b(a^3-27b)\neq 0
\end{equation}
for some $a,b \in k$.
If $(0,0)$ is neither a $2$-torsion nor a $3$-torsion, then by setting
$(x,y)\mapsto(u^2x,u^3y)$, where $u=b_3/b_2$, we find that the
equation of the elliptic curve becomes
$$
E:y^2+(1+a)xy+by=x^3+bx^2
$$
for some $a,b\in k$. This is called the Tate normal form (with the
starting rational point $P$) in the literature. By computing explicitly
the coordinates of $m(0,0)$ for $m=1,2,\ldots$, one can 
determine the relation between $a$ and $b$ such that $(0,0)$ is 
of order $N$.
%an $N$-torsion. 
As of now, the most extensive computation was due to
Baaziz \cite{Baaziz}, who determined the relation for $N$ up to $51$. 
(Similar computation, but in a smaller range of $N$, can be found in
several places, such as \cite{Reichert}.)
Such a relation can be taken as a defining equation
for the modular curve $X_1(N)$. When $k$ is a subfield of $\C$, Baaziz
also found expressions of $a$ and $b$ as modular functions in $\tau$,
where $\tau\in\H$ is a point such that
$(\C/(\Z\tau+\Z),1/N)\simeq(E,(0,0))$. The modular functions given in
\cite{Baaziz} were written in terms of the Weierstrass
$\wp$-function. For our purpose, we will give alternative expressions
using generalized Dedekind eta functions, which we recall in the next
subsection.

\begin{remark} \label{remark: uniqueness of Tate}
  Note that the Tate normal form is unique in the sense that if
$$
E:y^2+(1+a)xy+by=x^3+bx^2, \qquad
E':y^2+(1+a')xy+b'y=x^3+b'x^2
$$
are two elliptic curves in the Tate normal form such that
$(E,(0,0))\simeq(E',(0,0))$, then $a=a'$ and $b=b'$ (see Proposition
1.3 of \cite{Baaziz}).
\end{remark}

\subsection{Generalized Dedekind eta functions}
In this subsection, we recall the definition of generalized Dedekind
eta functions and their properties \cite{Yang-Dedekind}.

\begin{definition} \label{definition: generalized Dedekind}
  Let $N>1$ be an integer. For integers $g$ and
  $h$ not simultaneously congruent to $0$ modulo $N$, the generalized
  Dedekind eta function $E_{g,h}^{(N)}$ is defined by
  $$
  E_{g,h}^{(N)}(\tau)=q^{B(g/N)/2}\prod_{m=1}^\infty
  (1-\zeta^hq^{m-1+g/N})(1-\zeta^{-h}q^{m-g/N}),
  $$
  where $q=e^{2\pi i \tau}, \zeta=e^{2\pi i/N}$ and $B(x)=x^2-x+1/6$.
\end{definition}
  
In the next lemma, we describe the properties of generalized Dedekind
eta functions needed in our discussion.

\begin{lemma}[{\cite[Theorem 1]{Yang-Dedekind}}]
  \label{lemma: generalized Dedekind}
  Let $N>1$ be an integer and $g$ and $h$ be integers not
  simultaneously congruent to $0$ modulo $N$. Then we have
  \begin{equation} \label{eq: E translation}
  E_{g+N,h}^{(N)}=E_{-g,-h}^{(N)}
  =-\zeta^{-h}E_{g,h}^{(N)}, \qquad
  E_{g,h+N}^{(N)}=E_{g,h}^{(N)}.
  \end{equation}
%  where $\zeta=e^{2\pi i/N}$.
  Also, for $\gamma=\SM abcd\in\SL(2,\Z)$ with $c\ge0$, we have
  $$
  E_{g,h}^{(N)}(\tau+b)
  =e^{\pi ibB(g/N)}E_{g,bg+h}^{(N)}(\tau), \qquad
  \text{if }c=0,
  $$
  and
\begin{equation} \label{eq: E translation2}
  E_{g,h}^{(N)}(\gamma\tau)=\epsilon(a,b,c,d)e^{\pi i\delta}
  E_{g',h'}^{(N)}(\tau), \qquad
  \text{if }c> 0,
\end{equation}
  where
  \begin{equation*} %\label{eq: epsilon}
  \epsilon(a,b,c,d)=\begin{cases}
    e^{\pi i(bd(1-c^2)+c(a+d-3))/6}, &\text{if }c\text{ is odd},\\
    -ie^{\pi i(ac(1-d^2)+d(b-c+3))/6}, &\text{if }d\text{ is odd},
    \end{cases}
  \end{equation*}
  \begin{equation*} %\label{eq: delta}
  \delta=\frac{g^2ab+2ghbc+h^2cd}{N^2}-\frac{gb+h(d-1)}N,
  \end{equation*}
  and
  \begin{equation*} %\label{eq: g'h'}
  (g'\ h')=(g\ h)\M abcd.
  \end{equation*}
\end{lemma}

Using Lemma \ref{lemma: generalized Dedekind}, we easily obtain a
criterion when a product of $E_{0,h}^{(N)}(\tau)$ is a modular
function on $\Gamma_1(N)$. 

\begin{corollary} \label{corollary: modular criterion}
  Let $N>1$ be an integer. If $e_1,\ldots,e_{N-1}$ are
  integers satisfying
  $$
  \sum_{h=1}^{N-1}e_h\equiv0\bmod 12, \quad
  \sum_{h=1}^{N-1}he_h\equiv0\bmod 2, \quad
  \sum_{h=1}^{N-1}h^2e_h\equiv0\bmod 2N,
  $$
  then $\prod_{h=1}^{N-1}E_{0,h}^{(N)}(\tau)^{e_h}$ is a modular
  function on $\Gamma_1(N)$.
\end{corollary}

\begin{proof}
  Note that since $\SM0{-1}N0$ normalizes $\Gamma_1(N)$,
  $\prod_{h=1}^{N-1}E_{0,h}^{(N)}(\tau)^{e_h}$ is modular on $\Gamma_1(N)$
  if and only if $\prod_{h=1}^{N-1}E_{0,h}^{(N)}(-1/N\tau)^{e_h}$ is.
  By \eqref{eq: E translation} 
  and \eqref{eq: E translation2} applied to $\gamma = \SM{0}{-1}{1}{0}$,
%  By the transformation formula for generalized Dedekind eta functions, 
  the latter is equal to
  $\prod_{h=1}^{N-1}E_{h,0}^{(N)}(N\tau)^{e_h}$, up to a constant.
  Then Corollary 3 of \cite{Yang-Dedekind} shows that if
  $e_1,\ldots,e_{N-1}$ satisfy the stated conditions, then
  $\prod_{h=1}^{N-1}E_{h,0}^{(N)}(N\tau)^{e_h}$ is a modular function
  on $\Gamma_1(N)$. Consequently, 
  $\prod_{h=1}^{N-1} E_{0,h}^{(N)}(\tau)^{e_h}$
  is a modular function on $\Gamma_1(N)$. This proves the corollary.
\end{proof}

\begin{remark}
Note that by the Jacobi triple product identity, $E_{g,h}^{(N)}(\tau)$
is related to the Jacobi theta function $\vartheta_1(z|\tau)$ by
\begin{equation} \label{eq: E-theta}
  \vartheta_1\left(-\frac{g\tau+h}N\Big|\tau\right)
  =-ie^{-\pi ih/N}q^{-g^2/(2N^2)}\eta(\tau)E_{g,h}^{(N)}(\tau),
\end{equation}
where $\eta(\tau)=q^{1/24} \prod_{n=1}^\infty (1-q^n)$ 
is the classical Dedekind eta function.
In particular, we have
\begin{equation} \label{eq: E-theta 0}
  \begin{split}
    \vartheta_1(h/N|\tau)
    &=-ie^{\pi ih/N}\eta(\tau)E_{0,-h}(\tau)
    =-ie^{\pi ih/N}\eta(\tau)(-e^{-2\pi ih/N})E_{0,h}^{(N)}(\tau) \\
    &=ie^{-\pi ih/N}\eta(\tau)E_{0,h}^{(N)}(\tau),
  \end{split}
\end{equation}
which will be used frequently later on.
\end{remark}

\subsection{Hauptmodul}
%We now list our choice of $t$ in the following lemma.
We now list our choice of Hauptmodul (generator of $\C(X_1(N))$) in the following lemma.

\begin{lemma} \label{lemma: uniformizers}
For $N\in\{2,\ldots,10,12\}$, the function $t(\tau)$ in the table
below is a generator of the field of modular functions on
$X_1(N)$ such that $t(\tau)$ takes value $0$ and have a
$\Q$-rational Fourier expansion at the
cusp $0$. In the table we also 
list the values of $t(\tau)$ at other cusps.

\centerline{
\setlength\extrarowheight{6pt}
\begin{tabular}{c|l|l} \hline\hline
  $N$ & $t(\tau)$ & \text{values at other cusps} \\\hline
  \multirow{2}{*}{$2$}
      &\multirow{2}{*}{$\displaystyle
        \frac{\eta(\tau)^{24}}{64\eta(2\tau)^{24}}$}
                  & $t(\infty)=\infty$ \\
      & &\\ \hline
  \multirow{2}{*}{$3$}
      &\multirow{2}{*}{$\displaystyle
        \frac{\eta(\tau)^{12}}{27\eta(3\tau)^{12}}$}
                  & $t(\infty)=\infty$ \\
  & &\\ \hline
  \multirow{2}{*}{$4$}
      & \multirow{2}{*}{$\displaystyle\frac1{16}
        \frac{\eta(\tau)^{16}\eta(4\tau)^8}{\eta(2\tau)^{24}}$}
      & $t(1/2)=\infty$ \\
      & & $t(\infty)=1/16$ \\ \hline
  \multirow{2}{*}{$5$} & \multirow{2}{*}{$-\displaystyle
      \frac{E_{0,1}^{(5)}(\tau)^5}
        {E_{0,2}^{(5)}(\tau)^5}$}
      & $t(1/2)=\infty$ \\
      & & $1-11t(a/5)-t(a/5)^2=0$ \\ \hline
  \multirow{2}{*}{$6$} & \multirow{2}{*}{$\displaystyle\frac19
      \frac{\eta(\tau)^8\eta(6\tau)^4}{\eta(3\tau)^8\eta(2\tau)^4}$}
                  & $t(1/2)=1$, $t(1/3)=\infty$ \\
  & & $t(\infty)=1/9$ \\ \hline
  \multirow{2}{*}{$7$} &\multirow{2}{*}{$\displaystyle e^{4\pi i/7}
      \frac{E_{0,1}^{(7)}(\tau)^3}
      {E_{0,2}^{(7)}(\tau)^2
      E_{0,3}^{(7)}(\tau)}$}
                  & $t(1/2)=1$, $t(1/3)=\infty$ \\
  & & $1-8t(a/7)+5t(a/7)^2+t(a/7)^3=0$ \\ \hline
  \multirow{2}{*}{$8$} & \multirow{2}{*}{$\displaystyle
      i\frac{E_{0,1}^{(8)}(\tau)^2}
                         {E_{0,3}^{(8)}(\tau)^2}$}
                  & $t(1/3)=\infty$, $t(1/2)=1$, $t(1/4)=-1$ \\
  & & $1-6t(a/8)+t(a/8)^2=0$ \\ \hline
  \multirow{2}{*}{$9$} & \multirow{2}{*}{$\displaystyle e^{4\pi i/9}
      \frac{E_{0,1}^{(9)}(\tau)^2}
      {E_{0,2}^{(9)}(\tau)
      E_{0,4}^{(9)}(\tau)}$}
  & $t(1/2)=1$, $t(1/4)=\infty$, $1-t(a/3)+t(a/3)^2=0$\\
  & & $1-6t(a/9)+3t(a/9)^2+t(a/9)^3=0$ \\ \hline
  \multirow{2}{*}{$10$} & \multirow{2}{*}{$\displaystyle e^{4\pi i/10}
      \frac{E_{0,1}^{(10)}(\tau)
      E_{0,2}^{(10)}(\tau)}
      {E_{0,3}^{(10)}(\tau)E_{0,4}^{(10)}(\tau)}$}
                  & $t(1/3)=\infty$, $t(1/2)=1$, $t(1/4)=-1$ \\
  & & $1+t(a/5)-t(a/5)^2=0$, $1-4t(a/10)-t(a/10)^2=0$ \\ \hline
  \multirow{2}{*}{$12$} & \multirow{2}{*}{$\displaystyle e^{4\pi i/12}
       \frac{E_{0,1}^{(12)}(\tau)}
       {E_{0,5}^{(12)}(\tau)} \phantom{\Bigg|}$}
  & $t(1/5)=\infty$, $t(1/2)=1$, $1-t(a/3)+t(a/3)^2=0$ \\
  & & $1+t(a/4)^2=0$, $t(1/6)=-1$, $1-4t(a/12)+t(a/12)^2=0$
  \\ \hline\hline
\end{tabular}}
(Here in an expression such as $1-11t(a/5)-t(a/5)^2=0$, $a/5$ means
any cusp represented by $a/5$ for some integer $a$ with $(a,5)=1$.)
\end{lemma}

\begin{proof}
  For $N=2,3,4,6$, we have $\pm\Gamma_1(N)=\Gamma_0(N)$, so
  $\C(X_1(N))=\C(X_0(N))$, where $\C(X(\Gamma))$ denotes the field of
  modular functions on a congruence subgroup $\Gamma$ of $\SL(2,\Z)$.
  Generators of $\C(X_0(N))$ when $X_0(N)$ has genus $0$ are
  well-known. For example, they are given in Table 3 of
  \cite{Conway-Norton}. It is clear that $t(\infty)=\infty$ for
  $N=2,3$. Also, we may apply Propositions 3.2.1 and 3.2.8 of
  \cite{Ligozat} to check 
  that $t(\tau)=\eta(\tau)^{16}\eta(4\tau)^8/16\eta(2\tau)^{24}$ has
  only one simple zero at the cusp $0$ and is nonvanishing elsewhere
  on $X_0(4)$. Then using the transformation formula
  $\eta(-1/\tau)=\sqrt{\tau/i}\eta(\tau)$, we see that the
  coefficients in the Fourier expansion of $t(\tau)$ at $0$ are all
  rational numbers and the leading coefficient is $1$. The same thing
  can be said about the function $t(\tau)$ in the case $N=6$.

  For the other cases, we shall use results from \cite{Yang-Dedekind}.
  For the case $N=5$, applying the transformation formula in Lemma
  \ref{lemma: generalized Dedekind} with $\gamma=\SM0{-1}10$, we find
  $$
  t(-1/5\tau)=-\frac{E_{0,1}^{(5)}(-1/5\tau)^5}
  {E_{0,2}^{(5)}(-1/5\tau)^5}
  =\frac{E_{1,0}^{(5)}(5\tau)^5}
  {E_{2,0}^{(5)}(5\tau)^5}.
  $$
  According to \cite[Table 1]{Yang-Dedekind}, the function above
  generates $\C(X_1(5))$ and has a simple zero at the cusp $\infty$.
  Therefore, the function $t(\tau)$ itself generates $\C(X_1(5))$, has
  a simple zero at $0$ and is nonvanishing elsewhere, and has a
  $\Q$-rational Fourier expansion with leading coefficient $1$ at the
  cusp $0$. The other cases 
  $N=7,8,9,10,12$ can be explained in the same way.

  The values of $t(\tau)$ at cusps can be calculated using the
  transformation law for $E_{g,h}^{(N)}(\tau)$. Here as an example, we
  compute the value of $t(\tau)$ at $1/4$ in the case $N=10$.
  In the following computation, we omit the superscript $(10)$ from
  the notation $E_{g,h}^{(10)}(\tau)$.

  Let $\gamma=\SM1041$. By Lemma \ref{lemma: generalized Dedekind}, we
  have
  \begin{equation*}
    \begin{split}
      e^{4\pi i/10}\frac{E_{0,1}(\gamma\tau)
        E_{0,2}(\gamma\tau)}
      {E_{0,3}(\gamma\tau)E_{0,4}(\gamma\tau)}
      &=e^{4\pi i/10}
      \frac{e^{4\pi i/100}E_{4,1}(\tau)
        \cdot e^{16\pi i/100}E_{8,2}(\tau)}
      {e^{36\pi i/100}E_{12,3}(\tau)
        \cdot e^{64\pi i/100}E_{16,4}(\tau)} \\
      &=e^{-4\pi i/10}
      \frac{E_{4,1}(\tau)E_{8,2}(\tau)}
    {(-e^{-6\pi i/10})E_{2,3}(\tau)
      \cdot(-e^{-8\pi i/10})E_{6,4}(\tau)} \\
    &=-\frac{E_{4,1}(\tau)E_{8,2}(\tau)}
    {E_{2,3}(\tau)E_{6,4}(\tau)}
    =-1+\cdots.
    \end{split}
  \end{equation*}
  Therefore, the value of the modular function
  $e^{4\pi i/10}E_{0,1}(\tau)E_{0,2}(\tau)/{E_{0,3}(\tau)E_{0,4}(\tau)}$
  at the cusp $1/4$ is $-1$. We skip the details in other cases.
\end{proof}

We next give our choice of Hauptmodul of $X_1^0(N,2)$ when the
modular curve has genus $0$, i.e., when $N=2,4,6,8$.

\begin{lemma} \label{lemma: uniformizers 2}
  For $N\in\{2,4,6,8\}$, the function $u(\tau)$ in the
  table below is a generator of the field of modular functions on
  $X_1^0(N,2)$ such that $u(\tau)$ takes value $0$ and has a
  $\Q$-rational Fourier expansion at the
  cusp $0$. In the table we also list the values of $u(\tau)$ at the
  other cusps.

  \centerline{
\setlength\extrarowheight{6pt}
\begin{tabular}{c|l|l} \hline\hline
  $N$ & $u(\tau)$ & \text{values at other cusps} \\ \hline
  \multirow{2}{*}{$2$}
      &\multirow{2}{*}{$\displaystyle
        \frac{\eta(\tau/2)^{16}\eta(2\tau)^8}{\eta(\tau)^{24}}$}
                  & $u(\infty)=1$ \\
      & & $u(1)=\infty$ \\ \hline
  \multirow{2}{*}{$4$}
      &\multirow{2}{*}{$\displaystyle
        \frac{\eta(\tau/2)^{8}\eta(2\tau)^4}{\eta(\tau)^{12}}$}
                  & $u(1)=\infty$, $u(1/2)=-1$ \\
  & & $u(\infty)=1$ \\ \hline
  \multirow{2}{*}{$6$}
      & \multirow{2}{*}{$\displaystyle\frac13
        \frac{\eta(\tau/2)^4\eta(3\tau)^2}{\eta(3\tau/2)^{4}
        \eta(\tau)^2}$}
      & $u(1)=1$, $u(1/2)=-1$, $u(1/3)=-1/3$ \\
      & & $u(2/3)=\infty$, $u(\infty)=1/3$ \\ \hline
  \multirow{2}{*}{$8$} & \multirow{2}{*}{$\displaystyle
     e^{\pi i/4}\frac{E_{0,1}^{(8)}(\tau/2)}
        {E_{0,3}^{(8)}(\tau/2)}$}
      & $u(1)=1$, $1+2u(a/4)-u(a/4)^2=0$, $1+u(a/2)^2=0$ \\
      & & $u(1/3)=-1$, $u(2/3)=\infty$, $1-2u(a/8)-u(a/8)^2=0$
   \\ \hline \hline
\end{tabular}}
\smallskip

Moreover,
%$u(\tau)$ has the following properties.
%\begin{enumerate}
% \item
the function $u(\tau)$ is related to the generator $t(\tau)$
of $\C(X_1(N))$ given in Lemma \ref{lemma: uniformizers} by
\begin{equation} \label{eq: t-u}
  \begin{split}
    t=\begin{cases}
      4u/(1-u)^2, &\text{when }N=2, \\
      u/4(1+u)^2, &\text{when }N=4, \\
      u(1-u)/(1+3u), &\text{when }N=6, \\
      u(1-u)/(1+u), &\text{when }N=8.
    \end{cases}
  \end{split}
\end{equation}
%  \item We have
%    $$
%    u(\tau+1)=\begin{cases}
%      1/u(\tau), &\text{if }N=2, \\
%      1/u(\tau), &\text{if }N=4, \\
%      (1-u(\tau))/(1+3u(\tau)), &\text{if }N=6, \\
%      (1-u(\tau))/(1+u(\tau)), &\text{if }N=8. \end{cases}
%    $$
%  \end{enumerate}
\end{lemma}

\begin{proof} We observe that when $N=2,4,6$, we have
  $$
  \pm\Gamma_1^0(N,2)=\M1002^{-1}\Gamma_0(2N)\M1002.
  $$
  We can check that $u(2\tau)$ generates $\C(X_0(2N))$, takes value
  $0$ at the cusp $0$, and has a $\Q$-rational Fourier expansion at
  $0$. Therefore, $u(\tau)$ has the stated properties. The values
  of $u(\tau)$ at other cusps can be determined using the
  transformation law of $\eta(\tau)$.

  We now consider the case $N=8$. For $\gamma=\SM
  abcd\in\Gamma_1^0(8,2)$, we have
  $$
  \frac12\gamma\tau=\M a{b/2}{2c}d(\tau/2).
  $$
  Applying Lemma \ref{lemma: generalized Dedekind} with the matrix
  $\SM a{b/2}{2c}d$, we find that
  $$
  u(\gamma\tau)=e^{\pi i/4}\frac{E_{2c,d}^{(8)}(\tau/2)}
  {E_{6c,3d}^{(8)}(\tau/2)}
  =e^{\pi i/4}\frac{(-e^{-2\pi i/8})^{2c/8}E_{0,1}^{(8)}(\tau/2)}
  {(-e^{-6\pi i/8})^{2c/8}E_{0,3}^{(8)}(\tau/2)}
  =u(\tau).
  $$
  Thus, $u(\tau)$ is a modular function on $X_1^0(8,2)$. Using
  $q$-expansions, we can check that it is related to the Hauptmodul
  $t(\tau)$ for $X_1(8)$ specified in Lemma \ref{lemma: uniformizers}
  by $t(\tau)=u(\tau)(1-u(\tau))/(1+u(\tau))$. Since $\C(X_1^0(8,2))$
  is an extension of $\C(X_1(8))$ of degree $2$, $u(\tau)$ must be a
  generator of $\C(X_1^0(8,2))$. The values of $u(\tau)$ can be
  determined by the transformation law for generalized Dedekind eta
  functions. Finally, since $u(2\tau)^2=t(\tau)$, we see that
  $u(\tau)$ has $\Q$-rational Fourier expansions at the cusp $0$.

  The relation between $u(\tau)$ and $t(\tau)$ can
  be determined by using $q$-expansions or the covering $X_1^0(N,2)\to
  X_1(N)$.
%  The formula for $u(\tau+1)$ can be decided by how the
%  automorphism of $X_1^0(N,2)$ defined by $\tau\mapsto\tau+1$ acts on
%  the cusps. For example, in the case $N=6$, the automorphism fixes
%  the cusps $\infty$ and $1/2$, swaps the pair of $1/3$ and $2/3$,
%  and also the pair of $0$ and $1$. Then from the values of $u(\tau)$
%  at cusps we see that $u(\tau+1)=(1-u(\tau))/(1+3u(\tau))$. The
%  other cases can be proved in the same way.
\end{proof}

\subsection{Parameterization of elliptic curves with $N$-torsion}

In this subsection, we describe parameterizations of elliptic curves
with prescribed rational torsions.
%For simplicity, we assume that
%elliptic curves are defined over a subfield of $\C$, although the
%descriptions are also valid for any field, except possibly when the
%characteristic of the field is $2$ or $3$.

\begin{lemma} \label{lemma: ab in theta}
  Let $N\ge 4$ be an integer. Let
  $E:y^2+(1+a)xy+by=x^3+bx^2$ be an elliptic curve over a subfield $k$
  of $\C$ such that $(0,0)$ is of order $N$.
%  an $N$-torsion. 
  Assume that $\tau$ is a
  point in $\H$ such that
  $(E,(0,0))\simeq_\C(\C/(\Z\tau+\Z),1/N)$. Then we have
%  $$
%  a=\frac{E_{0,1}(\tau)^4E_{0,4}(\tau)}{E_{0,2}(\tau)^5}, \qquad
%  b=\frac{E_{0,1}(\tau)^5E_{0,3}(\tau)^3}{E_{0,2}(\tau)^8},
%  $$
%  where $E_{g,h}(\tau):=E_{g,h}^{(N)}(\tau)$. Moreover, the elliptic
%  curve $E$ is parameterized by
  \begin{equation} \label{eq: ab in theta}
    \begin{split}
      a&=a(\tau)=\frac{\vartheta_1(1/N)^4\vartheta_1(4/N)}
      {\vartheta_1(2/N)^5}
      =\begin{cases}
        0, &\text{if }N=4, \\
        \displaystyle e^{2\pi i/N}
        \frac{E_{0,1}^{(N)}(\tau)^4E_{0,4}^{(N)}(\tau)}
        {E_{0,2}^{(N)}(\tau)^5}, &\text{if }N\ge 5,\end{cases} \\
      b&=b(\tau)=\frac{\vartheta_1(1/N)^5\vartheta_1(3/N)^3}
      {\vartheta_1(2/N)^8}
      =e^{2\pi i/N}\frac{E_{0,1}^{(N)}(\tau)^5E_{0,3}^{(N)}(\tau)^3}
      {E_{0,2}^{(N)}(\tau)^8},
    \end{split}
  \end{equation}
  where we write the Jacobi theta function $\vartheta_1(z|\tau)$
  simply as $\vartheta_1(z)$. The functions $a(\tau)$ and $b(\tau)$
  are both modular functions on $\Gamma_1(N)$ with divisors supported
  on cusps.

  Moreover, the elliptic curve $E$ is
  parameterized by
  \begin{equation} \label{eq: x in theta}
  x(z)=-\frac{\vartheta_1(1/N)^4\vartheta_1(3/N)^2}
  {\vartheta_1(2/N)^6}\frac{\vartheta_1(z-1/N)\vartheta_1(z+1/N)}
  {\vartheta_1(z)^2}
  \end{equation}
  and
  \begin{equation} \label{eq: y in theta}
  y(z)=\frac{\vartheta_1(1/N)^7\vartheta_1(3/N)^3}
  {\vartheta_1(2/N)^{10}}\frac{\vartheta_1(z-1/N)^2\vartheta_1(z+2/N)}
  {\vartheta_1(z)^3}.
  \end{equation}
\end{lemma}

\begin{proof} According to Theorem 1.4 of \cite{Baaziz}, let
  $$
  a_1(z)=\frac{\wp''(z)}{\wp'(z)},\qquad
  a_2(z)=3\wp(z)-\frac14\frac{\wp''(z)^2}{\wp'(z)^2}, \qquad
  a_3(z)=\wp'(z),
  $$
  where $\wp(z):=\wp(z|\tau)$ is the Weierstrass $\wp$-function,
  and $\wt E$ be the elliptic curve
  $y^2+a_1(1/N)xy+a_3(1/N)y=x^3+a_2(1/N)x^2$. Then $(\wt
  E,(0,0))\simeq(\C/(\Z\tau+\Z),1/N)$ and the elliptic curve $\wt E$ is
  parameterized by
  $$
  x=\wp(z)-\wp(1/N), \qquad
  y=\frac12(\wp'(z)-a_1(1/N)x-a_3(1/N)).
  $$
  The elliptic curve is isomorphic to
  $$
  y^2+\frac{a_1(1/N)a_2(1/N)}{a_3(1/N)}xy
  +\frac{a_2(1/N)^3}{a_3(1/N)^2}y=x^3
  +\frac{a_2(1/N)^3}{a_3(1/N)^2}x^2
  $$
  through a simple change of variables with the parameterization of
  the new elliptic curve given by
%\RED{[To Yifan: $x(z)$ in the RHS of the second line is replaced by $\wp(z)-\wp(1/N)$
  \begin{equation} \label{eq: xy in P}
  \begin{aligned}
  &x(z)=\frac{a_2(1/N)^2}{a_3(1/N)^2}(\wp(z)-\wp(1/N)), \\
  &y(z)=\frac{a_2(1/N)^3}{2a_3(1/N)^3}(\wp'(z)-a_1(1/N)(\wp(z)-\wp(1/N))-a_3(1/N)).
  \end{aligned}
  \end{equation}
  By the uniqueness of Tate
  normal forms (see Remark \ref{remark: uniqueness of Tate}), we have
  $$
  a=\frac{a_1(1/N)a_2(1/N)-a_3(1/N)}{a_3(1/N)}, \qquad
  b=\frac{a_2(1/N)^3}{a_3(1/N)^2}.
  $$
  Now Proposition 1.5 of \cite{Baaziz} asserts that
  $$
  \div(a_2)=\sum_{\substack{3z=0\\z\neq0}}[z]
  -2\sum_{2z=0}[z], \qquad
  \div(a_3)=\sum_{\substack{2z=0\\z\neq0}}[z]-3[0],
  $$
  and
  $$
  \div(a_1a_2-a_3)=\sum_{\substack{4z=0\\2z\neq0}}[z]
  -3\sum_{2z=0}[z].
  $$
  Therefore,
  $$
  a_2(z)=C_1(\tau)\frac{\vartheta_1(3z)}
  {\vartheta_1(z)\vartheta_1(2z)^2}, \qquad
  a_3(z)=C_2(\tau)\frac{\vartheta_1(2z)}{\vartheta_1(z)^4},
  $$
  and
  $$
  a_1(z)a_2(z)-a_3(z)=C_3(\tau)
  \frac{\vartheta_1(4z)}{\vartheta_1(2z)^4}
  $$
  for some functions $C_j(\tau)$ of $\tau$. Considering the Laurent
  expansions $a_2(z)=\frac34z^{-2}+\cdots$, $a_3(z)=-2z^{-3}+\cdots$,
  and $a_1(z)a_2(z)-a_3(z)=-\frac14z^{-3}+\cdots$, we see that
  \begin{equation} \label{eq: a23}
  a_2(z)=\frac{\vartheta_1'(0)^2\vartheta_1(3z)}
  {\vartheta_1(z)\vartheta_1(2z)^2}, \qquad
  a_3(z)=-\frac{\vartheta_1'(0)^3\vartheta_1(2z)}{\vartheta_1(z)^4},
  \end{equation}
  and
  $$
  a_1(z)a_2(z)-a_3(z)=-\frac{\vartheta_1'(0)^3\vartheta_1(4z)}
  {\vartheta_1(2z)^4}.
  $$
  It follows that
  $$
  a(\tau)=\frac{\vartheta_1(1/N)^4\vartheta_1(4/N)}
  {\vartheta_1(2/N)^5},
  \qquad
  b(\tau)=\frac{\vartheta_1(1/N)^5\vartheta_1(3/N)^3}
  {\vartheta_1(2/N)^8}.
  $$
  By \eqref{eq: E-theta 0}, they can also be expressed as quotients of
  generalized Dedekind eta functions as claimed. Then applying
  Corollary \ref{corollary: modular criterion}, we easily see that
  $a(\tau)$ and $b(\tau)$ are modular functions on
  $\Gamma_1(N)$. Since $a(\tau)$ and $b(\tau)$ are products of
  generalized Dedekind eta functions, their zeros and poles must be at
  cusps. This proves the assertions about $a$ and $b$. We now find the
  expression for $x$ and $y$ in \eqref{eq: xy in P} in the Jacobi
  theta function.

  Considering the zeros and poles of $\wp(z)-\wp(1/N)$, we find
  $$
  \wp(z)-\wp(1/N)=C_4(\tau)\frac{\vartheta_1(z-1/N)\vartheta_1(z+1/N)}
  {\vartheta_1(z)^2}
  $$
  for some function $C_4(\tau)$ of $\tau$. Since
  $\wp(z)-\wp(1/N)=z^{-2}+\cdots$ and
  $\vartheta_1(z)=\vartheta_1'(0)z+\cdots$, we find that
  $C_4(\tau)=-\vartheta_1'(0)^2/\vartheta_1(1/N)^2$. It follows that
  $$
  x(z)=\frac{a_2(1/N)^2}{a_3(1/N)^2}(\wp(z)-\wp(1/N))
  =-\frac{\vartheta_1(1/N)^4\vartheta_1(3/N)^2}{\vartheta_1(2/N)^6}
  \frac{\vartheta_1(z-1/N)\vartheta_1(z+1/N)}{\vartheta_1(z)^2}.
  $$
  For the elliptic function $y(z)$ in \eqref{eq: xy in P}, we note
  that from the equation of the elliptic curve $E$, we see that $y(z)$
  has a double zero at $1/N$. Since it has three zeros and a
  triple pole at $0$, the other zero must be at $-2/N$. Thus,
  $$
  y(z)=C_5(\tau)\frac{\vartheta_1(z-1/N)^2\vartheta_1(z+2/N)}
  {\vartheta_1(z)^3}
  $$
  for some function $C_5(\tau)$ of $\tau$. Similar to the case of
  $x(z)$, by comparing the leading Laurent coefficients, we conclude
  the stated formula for $y(z)$.
%  that
%  $$
%  y(z)=\frac{\vartheta_1(1/N)^7\vartheta_1(3/N)^3}
%  {\vartheta_1(2/N)^{10}}\frac{\vartheta_1(z-1/N)^2\vartheta_1(z+2/N)}
%  {\vartheta_1(z)^3}.
%  $$
  This proves the lemma.
\end{proof}

\begin{proposition} \label{proposition: ENt}
  Let $N\in\{4,5,6,7,8,9,10,12\}$. Assume that $E$ is an elliptic
  curve over a subfield $k$ of $\C$ having a $k$-rational 
  point $P$ of order $N$.
%$N$-torsion $P$. 
  Let $\tau$ be a point in $\H$ such that
  $(E,P)\simeq_\C(\C/(\Z\tau+\Z),1/N)$.
  Then $(E,P)\simeq_k(E_{N,t(\tau)},(0,0))$, where $t(\tau)$ is the
  Hauptmodul of $X_1(N)$ given in Lemma \ref{lemma: uniformizers}
  and 
  \begin{equation} \label{eq: ENt}
  E_{N,t}: y^2+(1+a(t))xy+b(t)y=x^3+b(t)x^2
  \end{equation}
  with $(a,b)$ given by
  $$ \extrarowheight3pt
  \begin{array}{cl} \hline\hline
  N  & (a,b)  \\ \hline
  4 & (0,t)\\
  5 & (t,t) \\
  6 & (t,t-t^2) \\
  7 & (t-t^2,t(1-t)^2) \\
  8 & \displaystyle\left(\frac{t(1-t)}{1+t},
  \frac{t(1-t)}{(1+t)^2}\right) \\
  9 & (t(1-t)^2,t(1-t)^2(1-t+t^2))\\
  10 & \displaystyle\left(\frac{t(1-t)}{(1+t)(1+t-t^2)},
   \frac{t(1-t)}{(1+t)(1+t-t^2)^2}\right)
  \\
  12 & \displaystyle\left(\frac{t(1-t)(1-t+t^2)}{1+t},
  \frac{t(1-t)(1-t+t^2)(1+t^2)}{(1+t)^2}\right)
  \\ \hline\hline
  \end{array}
  $$
%  The value of $t(\tau)$ satisfies
%  $$
%  \begin{cases}
%    t(1-16t)\neq 0 &\text{if }N=4,\\
%    t(1-11t-t^2)\neq0 &\text{if }N=5, \\
%    t(1-t)(1-9t)\neq0, &\text{if }N=6, \\
%    t(1-t)(1-8t+5t^2+t^3)\neq0, &\text{if }N=7, \\
%    t(1-t)(1+t)(1-6t+t^2)\neq0, &\text{if }N=8, \\
%    t(1-t)(1-t+t^2)(1-6t+3t^2+t^3)\neq0, &\text{if }N=9, \\
%    t(1-t)(1+t)(1+t-t^2)(1-4t-t^2)\neq0, &\text{if }N=10, \\
%    t(1-t)(1+t)(1+t^2)(1-t+t^2)(1-4t+t^2)\neq0, &\text{if }N=12.
%  \end{cases}
%  $$
\end{proposition}

\begin{proof}
  Let $y^2+(1+a)xy+by=x^3+bx^2$ be the Tate normal form of $E$ with
  the starting point $P$. By Lemma \ref{lemma: ab in theta} and the
  uniqueness of the Tate normal form, $a$ and $b$ are
  equal to $a(\tau)$ and $b(\tau)$ defined by \eqref{eq: ab in theta}.
  By the same lemma, they are modular functions on
  $\Gamma_1(N)$. Thus, they are rational functions in $t(\tau)$. For
  example, in the case $N=5$, by \eqref{eq: E translation}, we have
  $$
  a(\tau)=e^{2\pi i/5}\frac{E_{0,1}^{(5)}(\tau)^4E_{0,4}^{(5)}(\tau)}
  {E_{0,2}^{(5)}(\tau)^5}
  =-\frac{E_{0,1}^{(5)}(\tau)^5}{E_{0,2}^{(5)}(\tau)^5}=t(\tau),
  $$
  and similarly, $b(\tau)=t(\tau)$. In general, by comparing Fourier
  expansions, we see that the expressions of $a$ and $b$ in terms of
  $t$ are as stated.
\end{proof}

\begin{remark}
  The parameterization of elliptic curves with $N$-torsion obtained in
  the proposition is valid for any field $k$. In fact, such universal
  elliptic curves are already known in the literature (see, e.g.,
  \cite{Husemoller,Kubert}). The point of the proposition is to give
  explicit parameterization in terms of Hauptmodul 
%  (generator of $\C(X_1(N))$) 
  of $X_1(N)$ when $k=\C$.

  Note that the proof of the proposition itself does not explain why
  $E_{N,t}$ is defined over $k$ if and only if $t\in k$.
  This can be seen from the relation
  $$
  t=\begin{cases}
    b, &\text{for }N=4,5, \\
    1-b/a, &\text{for }N=6, \\
    1-b/a, &\text{for }N=7, \\
    a/b-1, &\text{for }N=8, \\
    (a-a^2-b)/(a-b), &\text{for }N=9, \\
    (a-a^2-b)/a^2, &\text{for }N=10, \\
    -(a^3-ab+b^2)/(a-b)^2, &\text{for }N=12. \end{cases}
  $$
\end{remark}

\begin{proposition}
  Let $k$ be a subfield of $\C$.
    Let $N\in\{4,6,8\}$. Assume that $E$ is an elliptic curve over $k$
    having $k$-rational points $P, Q$ such that
    $P$ is of order $N$, $Q$ is of order $2$ and such that
    $Q\notin\gen P$. Then there exists $\tau\in\H$ such that
    $(E,P,Q)\simeq_k (E'_{N,u(\tau)},(0,0),Q')$,
%  for some $u\in k$ satisfying 
%  \begin{equation} \label{eq: singular u}
%  \begin{cases}
%  u(1-u)(1+u)\neq0, &\text{when }N=4, \\
%  u(1-u)(1+u)(1-3u)(1+3u)\neq0, &\text{when }N=6, \\
%  u(1-u)(1+u)(1+u^2)(1+2u-u^2)(1-2u-u^2)\neq0,
%  &\text{when }N=8, \end{cases}
%  \end{equation}
  where $u(\tau)$ is the Hauptmodul of $X_1^0(N,2)$ given in Lemma
  \ref{lemma: uniformizers 2} and the equation of $E_{N,u}'$ is given
  by
  \begin{equation} \label{eq: ENu}
  E'_{N,u}: y^2+(1+a(u))xy+b(u)y=x^3+b(u)x^2
  \end{equation}
  with
  $$
    (a,b)=\begin{cases}
      \displaystyle\left(0,\frac u{4(1+u)^2}\right),
      &\text{if }N=4, \\
      \displaystyle\left(\frac{u(1-u)}{1+3u},
          \frac{u(1-u)(1+u)^2}{(1+3u)^2}\right)
      &\text{if }N=6, \\
      \displaystyle\left(\frac{u(1-u)(1+u^2)}{(1+u)(1+2u-u^2)},
        \frac{u(1-u)(1+u^2)}{(1+2u-u^2)^2}\right),
      &\text{if }N=8, \end{cases}
  $$
    and
    \begin{equation} \label{eq: Q'}
    Q'=\begin{cases}
      \displaystyle\left(-\frac1{4(1+u)},
                  \frac1{8(1+u)^2}\right), &\text{if }N=4, \\

      \displaystyle\left(-\frac{(1-u)(1+u)^2}{4(1+3u)},
        \frac{(1-u)^2(1+u)^3}{8(1+3u)^2}\right), &\text{if }N=6, \\
      \displaystyle\left(-\frac{(1-u)(1+u)(1+u^2)}{4(1+2u-u^2)},
      \frac{(1-u)^2(1+u)(1+u^2)^2}{8(1+2u-u^2)^2}\right),
      &\text{if }N=8.
    \end{cases}
  \end{equation}
%    In addition, in the case $k$ is a subfield of $\C$, we have
%    $(E_{N,u(\tau)},(0,0),Q')\simeq_\C(\C/(\Z\tau+\Z),1/N,\tau/2)$,
%    where $u(\tau)$ is the uniformizer of $X_1^0(N,2)$ given in Lemma
%    \ref{lemma: uniformizers 2}.
\end{proposition}

\begin{proof}
  Let $\tau$ be a point in $\H$
  such that $(E,P)\simeq_\C(\C/(\Z\tau+\Z),1/N)$. The
  isomorphism yields $(E,P,Q)\simeq_\C(\C/(\Z\tau+\Z),1/N,\tau/2)$ or
  $(E,P,Q)\simeq_\C(\C/(\Z\tau+\Z),1/N,(\tau+1)/2)$. Changing $\tau$
  to $\tau-1$ if necessary, we may assume that the first case occurs.
  Then by Lemma \ref{lemma: ab in theta}, we have
  $(E,P,Q)\simeq_k(E',(0,0),Q')$, where $E'$ is the elliptic curve
  $y^2+a(\tau)xy+b(\tau)y=x^3+b(\tau)x^2$ with $a(\tau)$ and $b(\tau)$
  given by \eqref{eq: ab in theta} and $Q'=(x(\tau/2),y(\tau/2))$ with
  $x(z)$ and $y(z)$ defined by \eqref{eq: x in theta} and \eqref{eq: y
    in theta}, respectively. By the relation between $t(\tau)$ and
  $u(\tau)$ given in Lemma \ref{lemma: uniformizers 2} and Proposition
  \ref{proposition: ENt}, we see that the equation $E'$ is given by
  \eqref{eq: ENu}. Also, by \eqref{eq: E-theta}, $x(\tau/2)$ and
  $y(\tau/2)$ can be written as 
  \begin{equation*}
    \begin{split}
      x(\tau/2)&=-\frac{\vartheta_1(1/N)^4\vartheta_1(3/N)^2}
      {\vartheta_1(2/N)^6}\frac{\vartheta_1(\tau/2-1/N)
        \vartheta_1(\tau/2+1/N)}{\vartheta_1(\tau/2)^2} \\
      &=-e^{2\pi i/N}\frac{E_{0,1}^{(N)}(\tau)^4E_{0,3}^{(N)}(\tau)^2}
      {E_{0,2}^{(N)}(\tau)^6}\frac{E_{-N/2,1}^{(N)}(\tau)
        E_{-N/2,-1}^{(N)}(\tau)}{E_{-N/2,0}^{(N)}(\tau)^2}
    \end{split}
  \end{equation*}
  and
  \begin{equation*}
    \begin{split}
      y(\tau/2)&=\frac{\vartheta_1(1/N)^7\vartheta_1(3/N)^3}
      {\vartheta_1(2/N)^{10}}\frac{\vartheta_1(\tau/2-1/N)^2
        \vartheta_1(\tau/2+2/N)}{\vartheta_1(\tau/2)^2} \\
      &=e^{4\pi i/N}\frac{E_{0,1}^{(N)}(\tau)^7E_{0,3}^{(N)}(\tau)^3}
      {E_{0,2}^{(N)}(\tau)^{10}}\frac{E_{-N/2,1}^{(N)}(\tau)^2
        E_{-N/2,-2}^{(N)}(\tau)}{E_{-N/2,0}^{(N)}(\tau)^3}.
    \end{split}
  \end{equation*}
  Using \eqref{eq: E translation}, we simplify them to
  $$
  x(\tau/2)=-e^{2\pi i/N}\frac{E_{0,1}^{(N)}(\tau)^4E_{0,3}^{(N)}(\tau)^2}
  {E_{0,2}^{(N)}(\tau)^6}
  \frac{E_{N/2,1}^{(N)}(\tau)^2}{E_{N/2,0}^{(N)}(\tau)^2}
  $$
  and
  $$
  y(\tau/2)=e^{4\pi i/N}\frac{E_{0,1}^{(N)}(\tau)^7E_{0,3}^{(N)}(\tau)^3}
      {E_{0,2}^{(N)}(\tau)^{10}}\frac{E_{N/2,1}^{(N)}(\tau)^2
        E_{N/2,2}^{(N)}(\tau)}{E_{N/2,0}^{(N)}(\tau)^3}.
  $$
  Applying the transformation formula in Lemma \ref{lemma: generalized
    Dedekind}, we may check that the two functions of $\tau$ are both
  modular functions on $\Gamma_1^0(N,2)$. Thus, they are both rational
  functions of $u(\tau)$. comparing the Fourier expansions, we see
  that the expression of $Q'$ in terms of $u$ is \eqref{eq: Q'}.
\end{proof}

\begin{remark}
  The parameterization is valid for any field $k$ with
  $\operatorname{char}k\neq2$. To see this, we start from
  $(E,P)\stackrel{\phi}{\simeq}_k(E_{N,t},(0,0))$, $t\in k$, where
  $E_{N,t}$ is given by \eqref{eq: ENt} and $\phi$ is an isomorphism
  from $(E,P)$ to $(E_{N,t},(0,0))$. The elliptic curve $E_{N,t}$ has a full
  $2$-torsion subgroup if and only if the polynomial
  $x^3+bx^2+((1+a)x+b)^2/4$ splits completely over $k$. This implies
  that $t$ is of the form \eqref{eq: t-u} for some $u\in k$. Since
  $\phi(Q)$ is a $2$-torsion of $E_{N,t}$ not in $\gen{(0,0)}$, we have
  \begin{equation} \label{eq: phi(Q) 1}
  \phi(Q)=\begin{cases}
      \displaystyle\left(-\frac1{4(1+u)},
                  \frac1{8(1+u)^2}\right), &\text{if }N=4, \\

      \displaystyle\left(-\frac{(1-u)(1+u)^2}{4(1+3u)},
        \frac{(1-u)^2(1+u)^3}{8(1+3u)^2}\right), &\text{if }N=6, \\
      \displaystyle\left(-\frac{(1-u)(1+u)(1+u^2)}{4(1+2u-u^2)},
      \frac{(1-u)^2(1+u)(1+u^2)^2}{8(1+2u-u^2)^2}\right),
      &\text{if }N=8
    \end{cases}
  \end{equation}
  (which is the coordinates of $Q'$ in \eqref{eq: Q'}), or
  \begin{equation} \label{eq: phi(Q) 2}
  \phi(Q)=\begin{cases}
    \displaystyle\left(-\frac u{4(1+u)},
              \frac{u^2}{8(1+u)^2}\right), &\text{if }N=4, \\
      \displaystyle\left(-\frac{u(1+u)^2}{(1+3u)^2},
        \frac{u^2(1+u)^3}{(1+3u)^3}\right), &\text{if }N=6, \\
      \displaystyle\left(-\frac{u(1+u^2)}{(1+u)^2(1+2u-u^2)},
      \frac{u^2(1+u^2)^2}{(1+u)^3(1+2u-u^2)^3}\right),
      &\text{if }N=8. \end{cases}
  \end{equation}
  Observe that $t$ is invariant under the
  substitution
  $$
  u\longmapsto\begin{cases}
    1/u, &\text{if }N=4, \\
    (1-u)/(1+3u), &\text{if }N=6, \\
    (1-u)/(1+u), &\text{if }N=8, \end{cases}
  $$
  while the coordinates in \eqref{eq: phi(Q) 2} is changed to those in
  \eqref{eq: phi(Q) 1} under the substitution. 
  (In the case $k=\C$, this corresponds to the automorphism
  $\tau\mapsto\tau+1$ on $X_1^0(N,2)$.) This implies that
  $(E,P,Q)\simeq_k(E'_{N,u},(0,0),Q')$ for some $u\in k$.
\end{remark}

\section{Computation of the pairings}

In this section, we compute the pairing 
$\gen{P, Q}$ for $\Q$-rational torsion points $P, Q$
on an elliptic curve $E$ over $\Q$.
(By our convention, 
$\gen{P, Q}$ here means what is denoted by $\gen{P-O, Q-O}$ in \eqref{eq:pairing},
where $O \in E(k)$ is the identity element.)
The results are described in terms of 
universal polynomials $f_N(t), g_N^{(\nu)}(u)$.
They will be used to construct new modular curves
$X_1(N)_M^\pm$ and $X_1^0(N, 2)_{\ul{M}}^\pm$
in the next section.

\subsection{Cyclic case}
In the following discussion, for a nonzero rational function $F$ on an
elliptic curve over a field $k$ and a point $Q\in E(k)$, we let
$\LC_{Q,\pi}(F)$ denote the leading coefficient of $F$ at $Q$ with
respect to a uniformizer $\pi$ at $Q$.

\begin{proposition} \label{proposition: ENt lc}
  Let $N\ge 4$ be an integer. Let
  $E:y^2+(1+a)xy+by=x^3+bx^2$ be an elliptic curve over a subfield $k$
  of $\C$ such that $P:=(0,0)$ is of order $N$. Assume that $\tau$ is a
  point in $\H$ such that $(E,P)\simeq (\C/(\Z\tau+\Z),1/N)$.
  Let $F(x,y)$ be a rational function on $E$ such that
  $\div(F)=N(P-O)$. Choose $x$ and $y/x$ to be the local
  parameters for the points $P$ and $O$, respectively. Then
  $$
  \frac{\LC_{P,x}(F)}{\LC_{O,y/x}(F)}
  =\left(\frac{\vartheta_1(2/N)^2}{\vartheta_1(1/N)\vartheta_1(3/N)}
  \right)^N=\left(\frac{E_{0,2}^{(N)}(\tau)^2}
    {E_{0,1}^{(N)}(\tau)E_{0,3}^{(N)}(\tau)}\right)^N.
  $$
  Moreover, the last expression is a modular function on
  $\Gamma_1(N)$.
\end{proposition}

\begin{proof} Let $x(z)$ and $y(z)$ be defined by \eqref{eq: x in
    theta} and \eqref{eq: y in theta}, respectively. Then
  $$
  F(x(z),y(z))=C\left(\frac{\vartheta_1(z-1/N)}{\vartheta_1(z)}
    \right)^N
  $$
  for some constant $C$. We have
  \begin{equation*}
    \begin{split}
      \LC_{P,x}(F)&=C\lim_{z\to1/N}x(z)^{-N}
      \left(\frac{\vartheta_1(z-1/N)}{\vartheta_1(z)}\right)^N\\
      &=C\lim_{z\to1/N}\left(-\frac{\vartheta_1(2/N)^6\vartheta_1(z)^2}
        {\vartheta_1(1/N)^4\vartheta_1(3/N)^2\vartheta_1(z-1/N)
          \vartheta_1(z+1/N)}\frac{\vartheta_1(z-1/N)}{\vartheta_1(z)}
      \right)^N \\
      &=C\left(-\frac{\vartheta_1(2/N)^5}{\vartheta_1(1/N)^3
          \vartheta_1(3/N)^2}\right)^N.
    \end{split}
  \end{equation*}
  Similarly, we compute that
  \begin{equation*}
    \begin{split}
      \LC_{O,y/x}(F)
      &=C\lim_{z\to0}\left(\frac{y(z)}{x(z)}\right)^{-N}
      \left(\frac{\vartheta_1(z-1/N)}{\vartheta_1(z)}\right)^N\\
      &=C\lim_{z\to0}\left(
        -\frac{\vartheta_1(2/N)^4\vartheta_1(z)\vartheta_1(z+1/N)}
        {\vartheta_1(1/N)^3\vartheta_1(3/N)\vartheta_1(z-1/N)
          \vartheta_1(z+2/N)}\frac{\vartheta_1(z-1/N)}{\vartheta_1(z)}
      \right)^N \\
      &=C\left(-\frac{\vartheta_1(2/N)^3}{\vartheta_1(1/N)^2
          \vartheta_1(3/N)}\right)^N.
    \end{split}
  \end{equation*}
  It follows that
  $$
  \frac{\LC_{P,x}(F)}{\LC_{O,y/x}(F)}
  =\left(\frac{\vartheta_1(2/N)^2}{\vartheta_1(1/N)\vartheta_1(3/N)}
    \right)^N.
  $$
  By \eqref{eq: E-theta}, it can also be written as
  $$
  \left(\frac{E_{0,2}^{(N)}(\tau)^2}{E_{0,1}^{(N)}(\tau)
      E_{0,3}^{(N)}(\tau)}\right)^N.
  $$
  Applying Corollary \ref{corollary: modular criterion}, we see that
  the function above is a modular function on $\Gamma_1(N)$.
  This completes the proof of the proposition.
\end{proof}

We now compute the pairing $\gen{P, P}$ for an $N$-torsion rational point $P$.
We treat the cases $N=2$ and $3$ separately in \S \ref{sect:order2-3}.
For other $N$, the pairing is given in the next proposition.

\begin{proposition} \label{proposition: pairing 1}
  Let $k$ be a subfield of $\C$.
  Assume that $N\in\{4,5,6,7,8,9,10,12\}$.
  Let $E_{N,t}$, $t\in k$, be the elliptic curve defined by \eqref{eq: ENt}
  and put $P:=(0,0) \in E_{N,t}(k)_\Tor$. Then
  $$
  -\gen{P,P}=f_N(t)\otimes(1/N),
  $$
  where $f_N(t)$ is given by
  $$ \extrarowheight3pt
  \begin{array}{cl} \hline\hline
    N & f_N \\ \hline
    4 & t \\
    5 & t \\
    6 & t(1-t)^2 \\
    7 & t(1-t)^4 \\
    8 & t(1-t)^2(1+t)^4 \\
    9 & t(1-t)^4(1-t+t^2)^3 \\
    10 & t(1-t)^2(1+t)^8(1+t-t^2)^5 \\
    12 & t(1-t)^2(1-t+t^2)^3(1+t^2)^4(1+t)^6 \\ \hline\hline
  \end{array}
  $$
\end{proposition}

\begin{proof} By Proposition \ref{proposition: ENt lc} above, for the
  elliptic curve $E_{N,t(\tau)}$, we have
  $$
  -\gen{P,P}=\left(\frac{E_{0,1}^{(N)}(\tau)E_{0,3}^{(N)}(\tau)}
    {E_{0,2}^{(N)}(\tau)^2}\right)^N\otimes(1/N).
  $$
  Let $f(\tau)$ be the function in the expression above.
  Since $f(\tau)$ is a modular function on
  $\Gamma_1(N)$, it is a rational function of $t(\tau)$. By computing
  the Fourier expansions, we find that $f(\tau)=f_N(t(\tau))$ for
  $N=4,5,6,7,9$, and
  $$
  f(\tau)=\begin{cases}
    \displaystyle\frac{t(1-t)^2}{(1+t)^4}, &\text{when }N=8, \\
    \displaystyle\frac{t(1-t)^2}{(1+t)^2(1+t-t^2)^5},
    &\text{when }N=10, \\
    \displaystyle\frac{t(1-t)^2(1-t+t^2)^3(1+t^2)^4}{(1+t)^6},
    &\text{when }N=12. \end{cases}
  $$
  This yields the claimed formula for the pairing $-\gen{P,P}$.
\end{proof}

\subsection{Non-cyclic case}

We now compute the pairing 
$\gen{P, P}, \gen{Q, Q}, \gen{P, Q}$
for $P \in E(\Q)[N]$ and $Q \in E(\Q)[2]$
such that $Q \not\in \gen{P}$.
Again, we treat the case $N=2$ separately in \S \ref{sect:order2-3}.
The cases $N=4, 6, 8$ are treated in the following proposition.

%\begin{proposition} \label{proposition: pairing 2}
%  Let $k$ be a subfield of $\C$ and
%  $N\in\{4,6,8\}$. Let $E$ be an
%  elliptic curve defined by \eqref{eq: ENu}. Let
%  $D=(0,0)-(\infty)$ and $D'=(Q')-(\infty)$, where $Q'$ is the
%  $2$-torsion point given in \eqref{eq: Q'}. Then
%  $$
%  -\gen{D,D}=\begin{cases}
%    4u(1+u)^2\otimes(1/4), &\text{if }N=4, \\
%    u(1-u)(1+3u)^3(1+u)^4\otimes(1/6), &\text{if }N=6, \\
%    u(1-u)(1+u)(1+u^2)^2(1+2u-u^2)^4\otimes(1/8), &\text{if }N=8,
%    \end{cases}
%  $$
%  $$
%  \gen{D',D'}=\begin{cases}
%    (1-u)(1+u)\otimes(1/2), &\text{if }N=4, \\
%    (1-u)(1+u)(1-3u)(1+3u)\otimes(1/2), &\text{if }N=6, \\
%    (1-2u-u^2)(1+2u-u^2)\otimes(1/2), &\text{if }N=8, \end{cases}
%  $$
%  and
%  $$
%  \gen{D',D}=\begin{cases}
%    (1+u)\otimes(1/2), &\text{if }N=4, \\
%    (1-u)(1+3u)\otimes(1/2) &\text{if }N=6, \\
%    (1-u)(1+u)(1+u^2)(1+2u-u^2)\otimes(1/2),
%    &\text{if }N=8. \end{cases}
%  $$
%\end{proposition}

\begin{proposition} \label{proposition: pairing 2}
  Let $k$ be a subfield of $\C$ and
  $N\in\{4,6,8\}$. Let $E$ be an
  elliptic curve defined by \eqref{eq: ENu}, $P:=(0,0)$
  and $Q$ the
  $2$-torsion point given in \eqref{eq: Q'}. Then we have
\[
-\gen{P,P}=g_N^{(1)}(u) \otimes \frac{1}{N}, \quad
\gen{Q,Q}=g_N^{(2)}(u) \otimes \frac{1}{2}, \quad
\gen{P,Q}=g_N^{(3)}(u) \otimes \frac{1}{2},
\]
where $g_N^{(\nu)}(u)$ is given by 
  $$ \extrarowheight3pt
  \begin{array}{c c c c} \hline\hline
    N & 4 & 6 & 8 \\ \hline
    g_N^{(1)} & 4u(1+u)^2 & u(1-u)(1+3u)^3(1+u) & u(1-u)(1+u)(1+u^2)^2(1+2u-u^2)^4 \\
    g_N^{(2)} & (1-u)(1+u) & (1-u)(1+u)(1-3u)(1+3u) & (1-2u-u^2)(1+2u-u^2) \\
    g_N^{(3)} & (1+u) & (1-u)(1+3u) & (1-u)(1+u)(1+u^2)(1+2u-u^2) \\
 \hline\hline
  \end{array}
  $$
%  Let $k$ be a subfield of $\C$ and
%  $N\in\{4,6,8\}$. Let $E$ be an
%  elliptic curve defined by \eqref{eq: ENu}. Let
%  $D=(0,0)-(\infty)$ and $D'=(Q')-(\infty)$, where $Q'$ is the
%  $2$-torsion point given in \eqref{eq: Q'}. Then we have
%\[
%-\gen{D,D}=g_N^{(1)}(u) \otimes \frac{1}{N}, \quad
%\gen{D',D'}=g_N^{(2)}(u) \otimes \frac{1}{2}, \quad
%\gen{D,D'}=g_N^{(3)}(u) \otimes \frac{1}{2},
%\]
%where $g_N^{(\nu)}(u)$ is given by 
%  $$ \extrarowheight3pt
%  \begin{array}{c c c c} \hline\hline
%    N & 4 & 6 & 8 \\ \hline
%    g_N^{(1)} & 4u(1+u)^2 & u(1-u)(1+3u)^3(1+u) & u(1-u)(1+u)(1+u^2)^2(1+2u-u^2)^4 \\
%    g_N^{(2)} & (1-u)(1+u) & (1-u)(1+u)(1-3u)(1+3u) & (1-2u-u^2)(1+2u-u^2) \\
%    g_N^{(3)} & (1+u) & (1-u)(1+3u) & (1-u)(1+u)(1+u^2)(1+2u-u^2) \\
% \hline\hline
%  \end{array}
%  $$
%  $$
%  -\gen{D,D}=\begin{cases}
%    4u(1+u)^2\otimes(1/4), &\text{if }N=4, \\
%    u(1-u)(1+3u)^3(1+u)^4\otimes(1/6), &\text{if }N=6, \\
%    u(1-u)(1+u)(1+u^2)^2(1+2u-u^2)^4\otimes(1/8), &\text{if }N=8,
%    \end{cases}
%  $$
%  $$
%  \gen{D',D'}=\begin{cases}
%    (1-u)(1+u)\otimes(1/2), &\text{if }N=4, \\
%    (1-u)(1+u)(1-3u)(1+3u)\otimes(1/2), &\text{if }N=6, \\
%    (1-2u-u^2)(1+2u-u^2)\otimes(1/2), &\text{if }N=8, \end{cases}
%  $$
%  and
%  $$
%  \gen{D',D}=\begin{cases}
%    (1+u)\otimes(1/2), &\text{if }N=4, \\
%    (1-u)(1+3u)\otimes(1/2) &\text{if }N=6, \\
%    (1-u)(1+u)(1+u^2)(1+2u-u^2)\otimes(1/2),
%    &\text{if }N=8. \end{cases}
%  $$
\end{proposition}

\begin{proof}
  The formula for $-\gen{P,P}$ follows from Proposition
  \ref{proposition: pairing 1} and \eqref{eq: t-u}. We next compute
  $\gen{Q,Q}$.
  Let $(x_0,y_0)$ be the coordinates of $Q$ given in \eqref{eq:
    Q'}. We have $\div(x-x_0)=2(Q-O)$. Choose $y'=y+(1+a)x/2+b/2$ and
  $y/x$ to be uniformizers at $Q$ and $O$, respectively. Then
  $$
  \gen{Q,Q}=\frac{\LC_{Q,y'}(x-x_0)}
  {\LC_{O,y/x}(x-x_0)}\otimes(1/2).
  $$
  Now we have
  $$
  \frac{x-x_0}{(y')^2}=\left(x^2+\left(x_0+
      \frac{(1+a)^2}4+b\right)x-\frac{b^2}{4x_0}\right)^{-1}
  $$
  and consequently,
  $$
  \LC_{Q,y'}(x-x_0)=\left(2x_0^2+\frac{(1+a)^2x_0}4+bx_0
    -\frac{b^2}{4x_0}\right)^{-1}.
  $$
  Plugging in the actual values of $x_0$, $a$, and $b$, we find
  \begin{align*}
  \LC_{Q,y'}(x-x_0)=\begin{cases}
    \displaystyle\frac{1-u}{16(1+u)^3}, &\text{if }N=4, \\
    \displaystyle\frac{(1-u)^3(1+u)^3(1-3u)}{16(1+3u)^3},
    &\text{if }N=6, \\
    \displaystyle\frac{(1-2u-u^2)(1+u^2)^2(1-u)^4}{16(1+2u-u^2)^3},
    &\text{if }N=8, \end{cases}
  \end{align*}
  On the other hand, it is clear that $\LC_{O,y/x}(x-x_0)=1$.
  Then the formula for $\gen{Q,Q}$ follows.

  To compute $\gen{P,Q}$, we choose $x$ and $y/x$ to be the local
  parameters at $P$ and $O$, respectively. Then
  $\LC_{P,x}(x-x_0)$ is simply $-x_0$ and hence
  $\gen{P,Q}=(-x_0)\otimes(1/2)$, which yields the claimed formula
  for $\gen{P,Q}$.
\end{proof}

%\begin{remark}
%  Note that for $\tau\in\H$, the parametrization of 
%  $(E_{N,u(\tau)}',(0,0),Q')\simeq_\C(\C/(\Z\tau+\Z),1/N,\tau/2)$
%  given by \eqref{eq: xy in P} 
%\RED{[Yifan: \eqref{eq: xy in P} does not involve $2$-torsion point.
%Do you mean the formulas in p.16? I couldn't figure out.]}
%shows that $\LC_{Q,y'}(x-x_0)$ is
%  equal to
%  $$
%  \frac{a_2(1/N)^4}{a_3(1/N)^4}(\wp(\tau/2)-\wp(1/2))
%  (\wp(\tau/2)-\wp((\tau+1)/2),
%  $$
%  where $a_2(z)$ and $a_3(z)$ are defined by \eqref{eq: a23}.
%  Using the well-known formulas
%  $$
%  \wp(1/2)-\wp(\tau/2)=\pi^2\theta_3(\tau)^4, \qquad
%  \wp((\tau+1)/2)-\wp(\tau/2)=\pi^2\theta_2(\tau)^4,
%  $$
%  and
%  $$
%  \vartheta_1'(0)=\pi\theta_2(\tau)\theta_3(\tau)\theta_4(\tau),
%  $$
%  where
%  $$
%  \theta_2(\tau)=\frac{2\eta(2\tau)^2}{\eta(\tau)}, \qquad
%  \theta_3(\tau)=\frac{\eta(\tau)^5}{\eta(\tau/2)^2\eta(2\tau)^2},
%  \qquad
%  \theta_4(\tau)=\frac{\eta(\tau/2)^2}{\eta(\tau)}
%  $$
%  (see \cite[Pages 129 and 132]{McKean-Moll}), we find that, by \eqref{eq:
%    E-theta 0},
%  \begin{equation*}
%    \begin{split}
%      \LC_{Q,y'}(x-x_0)&=\left(\frac{\vartheta_1(1/N)^3
%          \vartheta_1(3/N)}
%      {\vartheta_1(2/N)^3\vartheta_1'(0)}\right)^4
%      \pi^4\theta_2(\tau)^4\theta_3(\tau)^4
%     =\left(\frac{E_{0,1}^{(N)}(\tau)^3E_{0,3}^{(N)}(\tau)\eta(\tau)}
%       {E_{0,2}^{(N)}(\tau)^3\theta_4(\tau)}\right)^4 \\
%     &=\left(\frac{E_{0,1}^{(N)}(\tau)^3E_{0,3}^{(N)}(\tau)\eta(\tau)^2}
%       {E_{0,2}^{(N)}(\tau)^3\eta(\tau/2)^2}\right)^4.
%    \end{split}
%  \end{equation*}
%  This formula is valid for all $N$.
%\end{remark}

\subsection{Torsion points of order two and three}\label{sect:order2-3}

We need to treat these cases with special care,
since $X_1(2), X_1(3), X_1^0(2, 2)$ are not fine moduli.
The following lemma is valid over a general field $k$
(subject to a condition on the characteristic).

\begin{lemma} \label{lemma: ENt}
  \begin{enumerate}
  \item Assume that the characteristic of $k$ is not $2$.
    If $E$ is an elliptic curve over $k$ with a $k$-rational point
    $P$ of order two, then $(E,P)$ is isomorphic to $(E_{2,t,a},(0,0))$
    for some $a,t\in k^\times$, where
    \begin{equation} \label{eq: E2ta}
      E_{2,t,a}:
      \begin{cases}
        \displaystyle y^2=x\left(x^2+ax+\frac{a^2t}{4(t+1)}\right),
          &\text{if }t\neq-1, \\
          y^2=x(x^2+a), &\text{if }t=-1.
      \end{cases}
    \end{equation}
    Moreover, two elliptic curves $(E_{2,t,a},(0,0))$ and
    $(E_{2,t',a'},(0,0))$ are isomorphic over $k$ if and only
    if $t=t'\neq-1$ and $a/a'$ is a square in $k^\times$
    or $t=t'=-1$ and $a/a'\in(k^\times)^4$ . 
  \item Assume that the characteristic of $k$ is not $3$.
    If $E$ is an elliptic curve over $k$ with a $k$-rational
    point $P$ of order $3$, then $(E,P)$ is isomorphic to $(E_{3,t},(0,0))$
    for some $t\neq-1$ in $k^\times$, where
    \begin{equation} \label{eq: E3t}
    E_{3,t}:y^2+xy+\frac t{27(t+1)}y=x^3
    \end{equation}
    or isomorphic to $(E_{3,-1,a},(0,0))$ for some $a\in k^\times$,
    where
    \begin{equation} \label{eq: E3a}
      E_{3,-1,a}:y^2+ay=x^3.
    \end{equation}
    Two elliptic curves $(E_{3,t},(0,0))$ and
    $(E_{3,t'},(0,0))$ with $t,t'\neq-1$ are isomorphic if and only
    if $t=t'$. Also, $(E_{3,-1,a},(0,0))\simeq(E_{3,-1,a'},(0,0))$
    if and only if $a/a'\in(k^\times)^3$. 
  \end{enumerate}
\end{lemma}
\begin{proof}
  Consider first the case $N=2$. 
We start with the equation \eqref{eq: E2}
and introduce a new parameter
  $t=4b/(a^2-4b)$ (so $b=a^2t/4(t+1)$). When $a\neq0$, we have
  $t\neq-1$ and the equation \eqref{eq: E2} in the new parameters is
  $y^2=x(x^2+ax+a^2t/4(t+1))$. When $a=0$, we change the notation $b$
  in \eqref{eq: E2} to $-a$ and get $y^2=x(x^2-a)$. Now it is easy to
  see that an isomorphism
  $\phi:(E_{2,t,a},(0,0))\to(E_{2,t',a'},(0,0))$ must be of the form
  $\phi(x,y)=(u^2x,u^3y)$ (c.f. \cite[Proposition 3.1]{sil}). From
  this, we immediately get the condition for $E_{2,t,a}$ and
  $E_{2,t',a'}$ to be isomorphic.

  For the case $N=3$, we similarly let $t=27b/(a^3-27b)$ (so
  $b=a^3t/27(t+1)$). Then \eqref{eq: E3} becomes
  $$
  y^2+axy+\frac{a^3t}{27(t+1)}y=x^3.
  $$
  If $a\neq0$, i.e., if $t\neq-1$, we make a change of variables
  $(x,y)\mapsto(a^2x,a^3y)$. Then the equation reduces to $E_{3,t}$.
  If $a=0$, i.e., if $t=-1$, then we change $b$ to $a$ and obtain
  $E_{3,-1,a}$. It is straightforward to verify the conditions for
  elliptic curves to be isomorphic. %We skip the details.
\end{proof}

\begin{remark}
When $k=\C$ and $\tau \in \C$, one can show
$(E_{2,t(\tau),a},(0,0))\simeq(\C/(\Z\tau+\Z),1/2)$
and $(E_{3,t(\tau)},(0,0))\simeq(\C/(\Z\tau+\Z),1/3)$ if
    $t(\tau)\neq-1$ and
    $(E_{3,-1,a},(0,0))\simeq(\C/(\Z\frac{1+\sqrt{-3}}{2}+\Z),1/3)$,
where $t(\tau)$ is given as in Lemma  \ref{lemma: uniformizers}.
We omit the proof, as we will not use them.
\end{remark}

We compute the pairing in the following lemma.

\begin{lemma} \label{lemma: pairing N=2,3}
  Let $k$ be a subfield of $\C$.
\begin{enumerate}
\item 
Let $E_{2,t,a}$ be an elliptic curve \eqref{eq: E2ta}
for some $a,t\in k^\times$, and put $P:=(0,0)$.
Then
\[
\gen{P,P}=
\begin{cases}
t(1+t) \otimes \frac{1}{2}, &\text{if }t\neq -1, \\
a \otimes \frac{1}{2}, &\text{if }t=-1.
\end{cases}
\]
\item 
Let $t \in k^\times$, and if $t=-1$ take also $a \in k^\times$.
Let $E$ be an elliptic curve given by either
$E_{3,t}$ from \eqref{eq: E3t} or $E_{3,-1, a}$ from \eqref{eq: E3a},
and put $P:=(0,0)$.
Then
\[
-\gen{P,P}=
\begin{cases}
t(1+t)^2 \otimes \frac{1}{3}, &\text{if }t\neq-1,\\
a \otimes \frac{1}{3}, &\text{if }t=-1. \\
\end{cases}
\]
\end{enumerate}
\end{lemma}

\begin{proof}
(1)
The divisor of the function $x$ on $E_{2,t,a}$ is $2(P-O)$. We choose $y$
  to be a uniformizer at $P$ and $y/x$ to be a uniformizer at $O$.
  Then
\begin{align*}
 & \lc_P(x, 2)=\frac x{y^2}\Big|_P \otimes \frac{1}{2}
  =\begin{cases}
   \displaystyle\frac{4(1+t)}{a^2t}\otimes \frac{1}{2}, &\text{if }t\neq-1,\\
   \displaystyle\frac{1}{a}\otimes\frac{1}{2}, &\text{if }t=-1, \end{cases} 
\\
 & \lc_O(x, 2)=\frac{x}{(y/x)^2}\Big|_O \otimes \frac{1}{2}=0.
\end{align*}
  Therefore, we have $\gen{P,P}=t(1+t)\otimes(1/2)$ when $t\neq-1$ and
  $\gen{P,P}=a\otimes(1/2)$ when $t=-1$.

(2) The divisor of the function $y$ on $E$ is $3(P-O)$.
We choose
  $x$ to be a uniformizer at $P$ and $x/y$ to be a uniformizer at
  $O$. Then
\begin{align*}
& \lc_P(y, 3)=\frac y{x^3}\Big|_P \otimes \frac{1}{3}=
  \begin{cases}
    \displaystyle\frac{27(1+t)}{t}\otimes\frac{1}{3}, &\text{if }t\neq-1, \\
    \displaystyle\frac{1}{a}\otimes\frac{1}{3}, &\text{if }t=-1, \end{cases}
\\
& \lc_O(y, 3)=\frac y{(y/x)^3}\Big|_O \otimes \frac{1}{3}=0.
\end{align*}
  It follows that $-\gen{P,P}=t(1+t)^2\otimes(1/3)$ when $t\neq-1$
  and $-\gen{P,P}=a\otimes(1/3)$ when $t=-1$.
\end{proof}

Next, we consider elliptic curves with full $2$-torsion points.
It is clear that, 
if $E$ is an elliptic curve over a field $k$ with $\ch k \not= 2$
 having two
    distinct $k$-rational $2$-torsion points $P$ and $Q$, then
    $(E,P,Q)\simeq(E'_{2,u,a},(0,0),(a,0))$ for some $u\neq0,1\in k$
    and $a\in k^\times$, where
    \begin{equation} \label{eq: E2ua}
      E_{2,u,a}': y^2=x(x-a)(x-au).
    \end{equation}

\begin{lemma} \label{lemma: pairing E2u}
  Let $E=E_{2,u,a}'$ be an elliptic curve from \eqref{eq: E2ua}
   over a subfield $k$ of $\C$,
  and put $P:=(0,0), Q:=(a,0)$. Then we have
\[
\gen{P,P}=u\otimes \frac{1}{2}, \quad \gen{Q,Q}=(1-u)\otimes \frac{1}{2},
\quad
\gen{P,Q}=a\otimes \frac{1}{2}.
\]
%  Let $E:y^2=x(x-a)(x-au)$, $au(1-u)\neq0$, be an elliptic
%  curve over $k$ with full $2$-torsion subgroup. 
%  $D=(0,0)-(\infty)$ and $D'=(a,0)-(0,0)$. Then
%  $\gen{D,D}=u\otimes(1/2)$, $\gen{D',D'}=(1-u)\otimes(1/2)$, and
%  $\gen{D',D}=a\otimes(1/2)$.
\end{lemma}

\begin{proof} 
We have $\div x=2(P-O)$ and $\div(x-a)=2(Q-O)$.
  Choose $y$, $y$, $x/y$ to be uniformizers at
  $P$, $Q$, and $O$, respectively. We have
  \begin{align*}
  &  \lc_P(x, 2)=\frac{x}{y^2}\Big|_P\otimes \frac{1}{2}
    =\frac1{(x-a)(x-au)}\Big|_P\otimes\frac12
    =u\otimes \frac{1}{2},\\
  &  \lc_Q(x-a, 2)=\frac{x-a}{y^2}\Big|_Q\otimes
    \frac{1}{2}
    =\frac1{x(x-au)}\Big|_Q\otimes\frac12
  =(1-u)\otimes \frac{1}{2},\\
  &  \lc_P(x-a,2)=(-a)\otimes \frac{1}{2},\\
  &  \lc_O(x,2)=\frac{x}{(y/x)^2}\Big|_O\otimes \frac{1}{2}=0,\\
  &  \lc_O(x-a,2)=\frac{x-a}{(y/x)^2}\Big|_O\otimes \frac{1}{2}=0.
  \end{align*}
  Then the formulas for the pairing follows.
\end{proof}

\begin{remark}
When $k=\C$, one can show
    $(\C/\Lambda_\tau,1/2,\tau/2)\simeq(E_{2,u(\tau),a}',(0,0),(a,0))$,  
    where $\tau \in \C$ and $u(\tau)$ is the uniformizer of $X_1^0(2,2)=X(2)$ given in
    Lemma \ref{lemma: uniformizers 2}.
We will not use this fact.
\end{remark}

\section{Modular curves $X_1(N)_M^\pm$ and $X_1^0(N, 2)_{\ul{M}}^{\pm}$}

Using the universal polynomials from
Propositions \ref{proposition: pairing 1} and \ref{proposition: pairing 2},
we are now able to construct new modular curves
that encode information on the values of the pairing.

\subsection{Modular curve $X_1(N)_M^\pm$}

\begin{definition}\label{def:X1NM}
Assume $N\in\{4,5,6,7,8,9,10,12\}$,
let $M$ be a positive divisor of $N$,
and $\epsilon \in \{ +, - \}$.
Let $f_N(t)$ be the function from Proposition \ref{proposition: pairing 1}.
%which we regard as a morphism $X_1(N) \to \P^1_z$ defined over $\Q$.
We define $X_1(N)_M^\epsilon$ to be the 
smooth projective model over $\Q$ of $\epsilon s^M = f_N(t)$.
In other words, $X_1(N)_M^\epsilon$ is the normalization of 
the fiber product 
\[
X_1'(N)_M^\epsilon := X_1(N)  \times_{\P^1_z} \P^1_s, \]
where 
$X_1(N) \to \P^1_z$ is 
the composition of the isomorphism 
$X_1(N) \cong \P^1_t$ from Lemma \ref{lemma: uniformizers}
followed by $\P^1_t \to \P^1_z,\ z=f_N(t)$,
and $\P^1_s \to \P^1_z$ is defined by $z=\epsilon s^M$.
We write $Y_1(N)_M^\epsilon \subset X_1(N)_M^\epsilon$ 
for the inverse image of $Y_1(N)$,
and put 
$X_1(N)_M^\pm := X_1(N)_M^+ \sqcup X_1(N)_M^-$, 
$Y_1(N)_M^\pm := Y_1(N)_M^+ \sqcup Y_1(N)_M^-$.
Here, as usual, we denote
the complement of cusps in $X_1(N)$ by $Y_1(N)$.
\end{definition}

The motivation for the definition will 
become clear in Corollary \ref{cor:criterion1} below.

\begin{remark}\label{rem:plus-minus}
\begin{enumerate}
\item 
If $k$ is a field of characteristic zero
containing an element $\zeta$ such that $\zeta^M=-1$,
there is an isomorphism $(X_1(N)_M^+)_k \cong (X_1(N)_M^-)_k$
given by $s \mapsto \zeta s$, where $( - )_k$ denotes the base change by $k/\Q$.
In particular, if $M$ is odd
we have $X_1(N)_M^+ \cong X_1(N)_M^-$ over $\Q$ 
(by taking $\zeta=-1$).
\item 
Observe that $t=0$ is a simple zero of $f_N(t)$ for any $N$.
It follows that $X_1'(N)_M^\e$ is non-singular 
at $(s, t)=(0, 0) \in X_1'(N)_M^\e(\Q)$,
and hence
there is a unique $\Q$-rational point on $X_1(N)_M^\e$ above it.
\item
For a positive divisor $M'$ of $M$, 
there is a finite morphism
\[ %\rho^{N}_{M, M'} : 
X_1(N)_{M}^\e \to X_1(N)_{M'}^\e
\]
induced by $s \mapsto s^{M/M'}$.
In particular, by taking $M'=1$ we get
\[
%\rho^{N}_{M}:=\rho^N_{M, 1} : 
X_1(N)_M^\e \to X_1(N)_1^\e= X_1(N).
\]
\end{enumerate}
\end{remark}

The normalization process is unnecessary for $Y_1(N)_N^\epsilon$.

\begin{lemma}\label{lem:normalization}
The canonical map $Y_1(N)_M^\epsilon \to X_1'(N)_M^\epsilon$
is an isomorphism onto its image.
\end{lemma}
\begin{proof}
Observe that the branch locus of 
$\P^1_s \to \P^1_z,  \ z=\epsilon s^M$ is 
contained in $\{ 0, \infty \} \subset \P^1_z$
(equal if $M>1$),
and that all points in $f_N^{-1}(\{ 0, \infty \}) \subset \P^1_t=X_1(N)$ 
are cusps by Lemma \ref{lemma: uniformizers}.
We conclude $X_1'(N)_M^\epsilon$ is smooth over $\Q$
outside cusps, whence the lemma.
\end{proof}

Recall that we have specified an isomorphism
$t : X_1(N) \cong \P^1$ in 
Lemma \ref{lemma: uniformizers}.
By the moduli interpretation,
there is a bijection between $Y_1(N)(\Q)$
and the set of isomorphism classes of pairs $(E, P)$ of 
an elliptic curve $E$ over $\Q$
and a $\Q$-rational point $P$ of order $N$.
By abuse of notation,
we write $(E, P) \in Y_1(N)(\Q)$ in this situation,
and write $t(E, P) \in \Q^\times \subset \P^1(\Q)$ 
for the point corresponding to $(E, P) \in Y_1(N)(\Q)$.
(We have $t(E, P) \not= 0, \infty$ since 
$t=0, \infty$ correspond to cusps on $X_1(N)$
by Lemma \ref{lemma: uniformizers}.)

\begin{lemma}\label{lem:Y1NM}
Let $N, M, \epsilon$ be as in Definition \ref{def:X1NM}.
\begin{enumerate}
\item 
There is a bijection between $Y_1(N)_M^\epsilon(\Q)$
and the set of triples  $(E, P, s)$ where
$(E, P) \in Y_1(N)(\Q)$ 
and $s \in \Q^\times$ such that $\epsilon s^M = f_N(t(E, P))$.
\item
Given $(E, P) \in Y_1(N)(\Q)$,
we have
\[
\gen{P, (N/M)P}=0 \Leftrightarrow
(E, P) \in \Im(Y_1(N)_M^\pm(\Q) \to Y_1(N)(\Q)).
\]
If $M$ is odd, one may replace $Y_1(N)_M^\pm(\Q)$ by $Y_1(N)_M^+(\Q)$.
\item
Suppose $M'|M$.
Then $X_1(N)_{M}^\epsilon \cong \P^1$ implies $X_1(N)_{M'}^\epsilon\cong \P^1$,
and $Y_1(N)_{M'}^\epsilon(\Q)=\emptyset$ implies $Y_1(N)_M^\epsilon(\Q)=\emptyset$.
\end{enumerate}
\end{lemma}
\begin{proof}
(1) follows from the definition and Lemma \ref{lem:normalization}.
Since $\gen{P, (N/M)P}=f_N(t(E, P)) \otimes (1/M)$
by Proposition \ref{proposition: pairing 1},
we get (2) from (1)
together with an elementary observation that
for $z \in \Q^\times$ one has $z \otimes (1/M)=0$ in $\Q^\times \otimes (\Q/\Z)$ 
if and only if $z=\pm s^M$ for some $s \in \Q^\times$.
To see (3), 
it suffices to note that
there is a finite morphism $X_1(N)_M^\epsilon \to X_1(N)_{M'}^\epsilon$
induced by $s \mapsto s^{M/M'}$.
We are done.
\end{proof}

\begin{corollary}\label{cor:criterion1}
For $(E, P) \in Y_1(N)(\Q)$,
we have $(E(\Q)_\Tor, E(\Q)_\Tor^\is) =(\gen{P}, \gen{(N/M)P})$
if and only if all of the following conditions are met:
\begin{enumerate}
\item 
$(E, P) \not\in \Im(Y_1(N\ell)(\Q) \to Y_1(N)(\Q))$
for any $\ell \in \{ 2, 3, 5 \}$ such that $N \ell \le 12$.
\item 
$(E, P) \not\in \Im(Y_1^0(N, 2)(\Q) \to Y_1(N)(\Q))$.
\item
$(E, P) \in \Im(Y_1(N)_M^\pm(\Q) \to Y_1(N)(\Q))$.
\item
$(E, P) \not\in \Im(Y_1(N)_{M'}^\pm(\Q) \to Y_1(N)(\Q))$
for any $M'>M$ such that  $M|M'|N$.
\end{enumerate}
Here $Y_1^0(N, 2)$ denotes the
complement of cusps in $X_1^0(N, 2)$.

\end{corollary}
\begin{proof}
(1) and (2) mean $E(\Q)_\Tor=\gen{P}$.
Under these conditions,
(3) and (4) mean $E(\Q)_\Tor^\is=\gen{(N/M)P}$
by Lemma \ref{lem:Y1NM} (2).
\end{proof}

These modular curves turn out to be
models of the classical modular curves.

\begin{theorem}\label{proposition: pairing 1-mod}
  Assume that $N\in\{4,5,6,7,8,9,10,12\}$ and
  let $M$ be a positive divisor of $N$.
  Denote by $X_1(MN,N)$ the modular curve over $\Q$ 
  associated to the congruence subgroup
    $$
    \Gamma_1(MN,N):= \Gamma_0(MN) \cap \Gamma_1(N)
    =\left\{\M abcd\in\SL(2,\Z) : \ 
      a,d\equiv1\bmod N, \ MN|c \right\},
    $$
  on which the cusp $0$ is $\Q$-rational.
  Then $X_1(N)_M^+$ is isomorphic to $X_1(MN,N)$ over $\Q$.
%
%
%  Assume that $N\in\{4,5,6,7,8,9,10,12\}$ and
%  let $M$ be a positive divisor of $N$ and $f_N(t)$ be the
%    function from Proposition \ref{proposition: pairing 1}. 
%    Then \RED{$s^M=f_N(t)$} is an equation for the
%    modular curve $X_1(MN,N)$ over $\Q$ associated to the congruence subgroup
%    $$
%    \Gamma_1(MN,N):=\left\{\M abcd\in\SL(2,\Z) : \ 
%      a,d\equiv1\bmod N, \ MN|c \right\}.
%    $$
%    Moreover, the equation $s^M=f_N(t)$ is singular only possibly at
%    cusps.
\end{theorem}

\begin{proof}
  Let
  $$
  s(\tau)=\left(\frac{E_{0,1}^{(N)}(\tau)E_{0,3}^{(N)}(\tau)}
  {E_{0,2}^{(N)}(\tau)^2}\right)^{N/M}.
  $$
  By Lemma \ref{lemma: generalized Dedekind}, for $\gamma=\SM
  abcd\in\Gamma_1(N)$, we have
  \begin{equation*}
    \begin{split}
      s(\gamma\tau)
      &=\left(e^{2\pi icd/N^2}\frac
        {E_{c,d}^{(N)}(\tau)E_{3c,3d}^{(N)}(\tau)}
      {E_{2c,2d}^{(N)}(\tau)^2}\right)^{N/M}
      =\left(e^{2\pi icd/N^2}\frac
        {E_{c,1}^{(N)}(\tau)E_{3c,3}^{(N)}(\tau)}
      {E_{2c,2}^{(N)}(\tau)^2}\right)^{N/M} \\
      &=\left(e^{2\pi icd/N^2}e^{-4\pi ic/N^2}\frac
        {E_{0,1}^{(N)}(\tau)E_{0,3}^{(N)}(\tau)}
        {E_{0,2}^{(N)}(\tau)^2}\right)^{N/M}
      =e^{-2\pi ic/MN}s(\tau).
    \end{split}
  \end{equation*}
  Therefore, $s(\gamma\tau)=s(\tau)$ if
  $\gamma\in\Gamma_1(MN,N)$. Also,
  $[\C(s(\tau),t(\tau)):\C(t(\tau))]=M=[\C(X_1(MN,N)):\C(X_1(N))]$.
  Consequently, $s(\tau)$ and $t(\tau)$ generate the field of modular
  functions on $X_1(MN,N)$. Moreover, using the transformation formula
  for generalized Dedekind eta functions, we see that the coefficients
  in the Fourier expansion of $s(\tau)$ at the cusp $0$ are all
  rational numbers. Thus, the relation $s^M=f_N(t)$ is an
  equation for $X_1(MN,N)$ over $\Q$. 
%  Finally, a singular point of $X_1'(N)_M^+$
%  the algebraic curve $s^M=f_N(t)$ occur only when $s=0$ and $s=\infty$. 
%  occur only when $s=0$ and $s=\infty$.
%  From the
%  definition of $s(\tau)$, such a point corresponds to a cusp of $X_1(MN,N)$.
  This completes the proof of the theorem.
\end{proof}

\begin{corollary}\label{cor:Y0MN-empty}
If $Y_0(MN)(\Q)=\emptyset$, then $Y_1(N)_M^+(\Q)=\emptyset$.
(According to Mazur and Kenku \cite{MazIsog, Kenku}
(see also \cite[Theorem 2.2]{BalMaz}),
one has $Y_0(N)(\Q) = \emptyset$ 
precisely when 
$N \ge 20$ and $N \not\in \{21, 25, 27, 37, 43, 67, 163 \}$.)
\end{corollary}
\begin{proof}
This is a direct consequence of Theorem \ref{proposition: pairing 1-mod},
since there is a morphism $X_1(MN, N) \to X_0(MN)$ over $\Q$
(as $\Gamma_1(MN, N) \subset \Gamma_0(MN)$ by definition).
\end{proof}

\begin{remark}
Let $k$ be a subfield of $\C$.
Then $X_1(MN, N)(k)$ parameterize
isomorphism classes of triples $(E, P, C)$ of an elliptic curve $E$ over $k$,
a $k$-rational point $P$ of order $N$,
and a cyclic subgroup $C$ of order $MN$ defined over $k$,
subject to the condition  $MC=\gen{P}$.
Given such a triple $(E, P, C)$,
take $\tau \in \H$ such that
$(E,P,C)\cong (\C/(\Z\tau+\Z), 1/N, \langle 1/MN \rangle)$,
and put  
\[s:=\left(\frac{E_{0,1}^{(N)}(\tau)E_{0,3}^{(N)}(\tau)}
{E_{0,2}^{(N)}(\tau)^2}\right)^{N/M}.
\]
Then, $(E,P,C)$ corresponds to $(E, P, s) \in X_1^0(N)_M^+(k)$
under the the isomorphism in Theorem \ref{proposition: pairing 1-mod}.
This can be read off from the proof of
Theorem \ref{proposition: pairing 1-mod}.
\end{remark}

\begin{proposition}\label{prop:rat-pt-Y1NM}
Assume that $N\in\{4,5,6,7,8,9,10,12\}$ and
let $M$ be a positive divisor of $N$, and $\epsilon \in \{ +, - \}$.
\begin{enumerate}
\item 
If one of the following conditions holds,
$X_1(N)_M^\epsilon$ is isomorphic to $\P^1$ over $\Q$.
%(hence $Y_1(N)_M^\epsilon(\Q)$ contains infinitely many points).
\begin{enumerate}
\item $M=1$.
\item $(N, M) \in \{ (4,2), (4, 4), (5, 5), (6, 2), (6, 3), (8, 2) \}$.
\end{enumerate}
\item
If neither of (a), (b) is satisfied, then we have $Y_1(N)_M^\epsilon(\Q)=\emptyset$.
\end{enumerate}
\end{proposition}
\begin{proof}
(1) 
Since $X_1(N)_M^\epsilon$ always has a $\Q$-rational point by Remark \ref{rem:plus-minus} (2),
it suffices to show that its genus is zero.
For this, we may assume $\epsilon=+$ by Remark \ref{rem:plus-minus} (1),
but then 
Theorem \ref{proposition: pairing 1-mod} shows 
$X_1(N)_M^+$ is isomorphic to $X_1(MN, N)$,
which has genus zero.
(This can also be seen easily from the defining equation $s^M=f_N(z)$.)

(2) 
We first note that
the cases 
$(N, M) \in \{ (8, 8), (9, 9), 
(10, 10), (12, 4), (12, 6), (12, 12) \}$ 
are reduced to the case $(N, M')$ for
some smaller $M'|N$ by Lemma \ref{lem:Y1NM} (3).
Next, if $\epsilon=+$ or $M$ is odd,
we can apply Remark \ref{rem:plus-minus} (1)
and Corollary \ref{cor:Y0MN-empty}
to deduce $Y_1(N)_M^\epsilon(\Q)=\emptyset$.
This covers the cases
$(N, M) \in \{ (7, 7), (10, 5), (12, 3) \}$ with $\epsilon \in \{+, -\}$,
and $(N, M) \in \{ (6,6), (8,4), (10,2), (12,2) \}$ with $\epsilon=-$.

This leaves us with the cases
$(N, M) \in \{ (6,6), (8, 4), (9,3), (10, 2), (12, 2) \}$ with $\epsilon=-$.
(For $(N, M)=(9,3)$, the value of $\epsilon$ does not matter
because of Remark \ref{rem:plus-minus} (1).)
In all cases, there exists 
an isomorphism $\alpha : E \to X_1(N)_M^\epsilon$,
where
$E$ is an elliptic curve over $\Q$
with finitely many $\Q$-rational points,
as exhibited in the following table,
where we set
$h_1(x)=(1-x)(1+x)^4(1+x-x^2)^2$ and
%$h_2(x)=(1+x)(1+x+x^2)(1+x^2)^2(1-x)^3$.
$h_2(x)=(1-x)(1-x+x^2)(1+x^2)^2(1+x)^3$.
  $$ \extrarowheight3pt
  \begin{array}{c|c|c|c|c} \hline\hline
    (N, M) & E & E(\Q) \setminus \{ O \} & 
(s, t)=\alpha(x,y) & \text{LMFDB} \\ \hline
(6,6) & y^2=x^3-1 & (1,0) &
\left(-\frac{xy}{1-x^3},\frac1{1-x^3}\right)
& \text{144.a3}
\\
(8,4) & y^2=x^3+x & (0,0) &
\left(\frac y{x^2}(1-\frac{1}{x^2}),-\frac1{x^2}\right)
& \text{64.a4}
\\
(9,3) & y^2+y=x^3 & (0,0), (0,-1) &
\left(xy(y^2+y+1),y+1\right)
& \text{27.a4}
\\
(10,2) & y^2=x^3-x^2-x & (0,0) &
\left(yh_1(x),x\right)
& \text{80.b3}
\\
(12,2) & y^2=x^3-x^2+x & (0,0), (1,\pm1) &
\left(yh_2(x),x\right)
& \text{24.a5}
%(12,2) & y^2=x^3+x^2+x & (0,0) &
%\left(yh_2(x),x\right)
%& \text{48.a5}
\\
 \hline\hline
  \end{array}
  $$
%
%
%
%  $$ \extrarowheight3pt
%  \begin{array}{c|c|c|c} \hline\hline
%    (N, M) & E & E(\Q) \setminus \{ O \} & 
%(s, t)=\alpha(x,y) \\ \hline
%(6,6) & y^2=x^3-1 & (1,0) &
%\left(-\frac{xy}{1-x^3},\frac1{1-x^3}\right)
%\\
%(8,4) & y^2=x^3+x & (0,0) &
%\left(\frac y{x^2}(1-\frac{1}{x^2}),-\frac1{x^2}\right)
%\\
%(9,3) & y^2+y=x^3 & (0,0), (0,-1) &
%\left(xy(y^2+y+1),y+1\right)
%\\
%(10,2) & y^2=x^3-x^2-x & (0,0) &
%\left(yh_1(x),x\right)
%\\
%(12,2) & y^2=x^3+x^2+x & (0,0) &
%\left(yh_2(x),x\right)
%\\
% \hline\hline
%  \end{array}
%  $$
%
%
(The last column indicates the labels of elliptic curves
in the database \cite{LMFDB},
where the information on $E(\Q)$ are available.)
Since all points of $\alpha(E(\Q))$ are cusps 
by Lemma \ref{lemma: uniformizers},
we get $Y_1(\Q)^\epsilon_M=\emptyset$ in these cases as well.
We are done.
%
%We first assume
%$(N, M, \epsilon)=(6,6, -)$
%or $(N, M) \in \{ (8, 4), (10, 2), (12, 2) \}$.
%Then there exists 
%an isomorphism $\alpha : E \to X_1(N)_M^\epsilon$,
%where
%$E$ is an elliptic curve over $\Q$
%with finitely many $\Q$-rational points,
%as exhibited in the following table,
%in which  we set
%$h_1(x)=(1-x)(1+x)^4(1+x-x^2)^2$ and
%$h_2(x)=(1-x)(1-x+x^2)(1+x^2)^2(1+x)^3$.
%  $$ \extrarowheight3pt
%  \begin{array}{c|c|c|c} \hline\hline
%    (N, M, \epsilon) & E & E(\Q) \setminus \{ \infty \} & 
%(s, t)=\alpha(x,y) \\ \hline
%(6,6, -) & y^2=x^3-1 & (1,0) &
%\left(-\frac{xy}{1-x^3},\frac1{1-x^3}\right)
%\\
%(8,4,+) & y^2=x^3-x & (0,0), (\pm 1, 0) &
%\left(\frac y{x^2}(1+\frac{1}{x^2}),\frac1{x^2}\right)
%\\
%(8,4,-) & y^2=x^3+x & (0,0) &
%\left(\frac y{x^2}(1-\frac{1}{x^2}),-\frac1{x^2}\right)
%\\
%(10,2,+) & y^2=x^3+x^2-x & (0,0), (-1, \pm 1), (1, \pm 1) &
%\left(yh_1(-x),-x\right)
%\\
%(10,2,-) & y^2=x^3-x^2-x & (0,0) &
%\left(yh_1(x),x\right)
%\\
%(12,2,+) & y^2=x^3-x^2+x & (0,0), (1, \pm 1) &
%\left(yh_2(x),x\right)
%\\
%(12,2,-) & y^2=x^3+x^2+x & (0,0) &
%\left(yh_2(-x),-x\right)
%\\
% \hline\hline
%  \end{array}
%  $$
%Since all points of $\alpha(E(\Q))$ are cusps 
%by Lemma \ref{lemma: uniformizers},
%we get $Y_1(\Q)^\epsilon_M=\emptyset$ in these cases.
\end{proof}

\subsection{Modular curve $X_1^0(N, 2)_{\ul{M}}^\pm$}

\begin{definition}\label{def:X10NM}
Assume $N\in\{4,6,8\}$.
We set
\begin{align*}
&\sM_N:= \{ \ul{M}=(M_1, M_2, M_3) \mid M_1 | N, \ M_2, M_3 \in \{ 1, 2 \} \},\ 
\\
&\sS:= \{ +, - \}^3=\{ \ul{\e}=(\e_1, \e_2, \e_3) \mid \e_1, \e_2, \e_3 \in \{ +, - \} \}.
\end{align*}
Let $g_N^{(\nu)}(u) \ (\nu=1, 2, 3)$ 
be the function from Proposition \ref{proposition: pairing 2}.
%which we regard as a morphism $X_1(N) \to \P^1_z$ defined over $\Q$.
For $\ul{M} \in \sM_N$ and $\ul{\epsilon} \in \sS$,
we define $X_1^0(N, 2)_{\ul{M}}^{\ul{\epsilon}}$ to be the 
smooth projective model over $\Q$ of the curve defined by
$\epsilon_i v_\nu^{M_\nu} = g_N^{(\nu)}(u) \ (\nu=1,2,3)$.
In other words, 
$X_1^0(N, 2)_{\ul{M}}^{\ul{\epsilon}}$
is the normalization of the fiber product 
\[
{X_1^0}'(N, 2)_{\ul{M}}^{\ul{\epsilon}}
:= X_1^0(N, 2) \times_{(\P^1_{w_1} \times \P^1_{w_2} \times \P^1_{w_3})}
  (\P^1_{v_1} \times \P^1_{v_2} \times \P^1_{v_3}),
\]
where 
$X_1^0(N, 2) \to \P^1_{w_1} \times \P^1_{w_2} \times \P^1_{w_3}$ 
is the composition of the isomorphism 
$X_1^0(N, 2) \cong \P^1_u$ from Lemma \ref{lemma: uniformizers 2}
followed by $\P^1_u \to \P^1_{w_1} \times \P^1_{w_2} \times \P^1_{w_3},\ 
w_\nu=g_N^{(\nu)}(u)\ (\nu=1, 2, 3)$,
and 
$\P^1_{v_1} \times \P^1_{v_2} \times \P^1_{v_3} \to \P^1_{w_1} \times \P^1_{w_2} \times \P^1_{w_3}$ 
is defined by $w_\nu=v_\nu^{M_\nu} \ (\nu=1, 2, 3)$.
(Here $\times$ denotes the fiber product over $\Spec \Q$.)
We write 
$Y_1^0(N, 2)_{\ul{M}}^{\ul{\e}} \subset X_1^0(N, 2)_{\ul{M}}^{\ul{\e}}$
for the inverse image of 
the complement $Y_1^0(N, 2)$ of cusps in $X_1^0(N, 2)$.
Also, we write
$X_1^0(N, 2)_{\ul{M}}^\pm$
for the disjoint union of
$X_1^0(N,2)_{\ul{M}}^{\ul{\e}}$
over $\ul{\e} \in \sS$,
and similarly for 
$Y_1^0(N, 2)_{\ul{M}}^\pm$.
\end{definition}

\begin{remark}\label{rem:plus-minus2}
\begin{enumerate}
\item
Observe that
$u=0$ is a simple zero of $g_N^{(1)}(u)$, 
and $g_N^{(\nu)}(0)=1 \in \P^1_{w_\nu}$ is not a branch point of 
$\P^1_{v_\nu} \to \P^1_{w_\nu}, w_\nu=v_\nu^{M_\nu}$ for $\nu=2, 3$.
It follows that, if $\e_2=\e_3=+$, then
${X_1^0}'(N)_{\ul{M}}^{\ul{\e}}$ is non-singular 
at $(u, v_1, v_2, v_3)=(0, 0, 1, 1) \in {X_1^0}'(N,2)_{\ul{M}}^{\ul{\e}}(\Q)$,
and hence
there is a unique $\Q$-rational point on $X_1^0(N,2)_{\ul{M}}^{\ul{\e}}$ above it.
%(It can happen $X_1^0(N,2)_{\ul{M}}^{\ul{\e}}(\Q)=\emptyset$ if $\e_2$ or $\e_3$ is $-$,
%as we will encounter in the proof of Proposition \ref{prop:rat-pt-Y10NM} (2),
%case of $N=8, \ul{\e}=(1,2,1)$.)
\item 
For 
$\ul{M}=(M_1, M_2, M_3), \ul{M}'=(M_1', M_2', M_3') \in \sM_N$,
we write $\ul{M}'|\ul{M}$ if $M_\nu'|M_\nu$ holds for all $\nu=1, 2, 3$.
In this case, there is a finite morphism
\[ %\rho^{0N}_{\ul{M}, \ul{M}'} : 
X_1^0(N, 2)_{\ul{M}}^{\ul{\e}} \to X_1^0(N, 2)_{\ul{M}'}^{\ul{\e}}
\]
induced by $v_\nu \mapsto v_\nu^{M_\nu/M_\nu'} \ (\nu=1, 2, 3)$.
By taking $\ul{M}'=(1,1,1)$ we get
\[
%\rho^{0N}_{\ul{M}}:=\rho^N_{\ul{M}, (1,1,1)} : 
X_1^0(N, 2)_{\ul{M}}^{\ul{\e}} \to X_1^0(N, 2)_{(1,1,1)}^{\ul{\e}}= X_1^0(N, 2).
\]
\end{enumerate}
\end{remark}

\begin{lemma}\label{lem:normalization2}
The canonical map 
$Y_1^0(N, 2)_{\ul{M}}^{\ul{\e}} \to 
{X_1^0}'(N, 2)_{\ul{M}}^{\ul{\e}}$
is an isomorphism onto its image.
\end{lemma}
\begin{proof}
This is shown by the same argument as Lemma \ref{lem:normalization}.
\end{proof}

Recall that we have specified an isomorphism
$u : X_1^0(N, 2) \cong \P^1$ in 
Lemma \ref{lemma: uniformizers 2}.
By the moduli interpretation,
there is a bijection between $Y_1^0(N, 2)(\Q)$
and the set of isomorphism classes of triples $(E, P, Q)$ of 
an elliptic curve $E$ over $\Q$,
a $\Q$-rational point $P$ of order $N$,
and a $\Q$-rational point $Q$ of order $2$
such that $Q \not\in \gen{P}$.
By abuse of notation,
we write $(E, P, Q) \in Y_1^0(N, 2)(\Q)$ in this situation,
and write $u(E, P, Q) \in \Q^\times \subset \P^1(\Q)$ 
for the point corresponding to $(E, P, Q) \in Y_1^0(N, 2)(\Q)$.
%(We have $t(E, P, Q) \not= 0, \infty$ since 
%$t=0, \infty$ correspond to cusps on $X_1^0(N, 2)$
%by Lemma \ref{lemma: uniformizers 2}.)

\begin{lemma}\label{lem:Y10NM}
Let $N, \ul{M}, \ul{\e}$ be as in Definition \ref{def:X10NM}.
\begin{enumerate}
\item 
There is a bijection between 
$Y_1^0(N, 2)_{\ul{M}}^{\ul{\e}}(\Q)$
and the set of sextuple $(E, P, Q, v_1, v_2, v_3)$ where
$(E, P, Q) \in Y_1^0(N, 2)(\Q)$ 
and $v_\nu \in \Q^\times$ such that $\epsilon_\nu v_\nu^{M_\nu} = g_N^{(\nu)}(u(E, P, Q))$
$(\nu=1, 2, 3)$.
\item
Given $(E, P, Q) \in Y_1^0(N, 2)(\Q)$,
we have 
\[
\gen{P, \frac{N}{M_1}P}=\gen{Q, \frac{2}{M_2}Q}=\gen{P, \frac{2}{M_3}Q}=0
\]
if and only if
$(E, P, Q)$ belongs to the image of
the canonical map 
$Y_1^0(N, 2)_{\ul{M}}^\pm(\Q)
\to 
Y_1^0(N, 2)(\Q)$.
\item
Suppose $\ul{M}'|\ul{M}$.
If 
$X_1^0(N, 2)_{\ul{M}}^{\ul{\e}}\cong \P^1$,
then 
$X_1^0(N, 2)_{\ul{M}'}^{\ul{\e}}\cong \P^1$.
If 
$Y_1^0(N, 2)_{\ul{M}'}^{\ul{\e}}(\Q)=\emptyset$,
then 
$Y_1^0(N, 2)_{\ul{M}}^{\ul{\e}}(\Q)=\emptyset$.
\end{enumerate}
\end{lemma}
\begin{proof}
Again, this is shown by the same argument as Lemma \ref{lem:Y1NM}.
\end{proof}

\begin{corollary}\label{cor:criterion2}
Suppose $(E, P, Q) \in Y_1^0(N,2)(\Q)$
and $E(\Q)_\Tor=\gen{P, Q}$.
Let $M$ be a positive divisor of $N$,
and put $m := \gcd(N/M, 2)$.
Then we have
\begin{align*}
& \frac{N}{M} P \in E(\Q)_\Tor^\is \Leftrightarrow 
(E, P, Q) \in \Im(Y_1^0(N,2)_{(M,1,m)}^\pm(\Q) \to Y_1^0(N,2)(\Q)),
\\
& Q \in E(\Q)_\Tor^\is \Leftrightarrow 
(E, P, Q) \in \Im(Y_1^0(N,2)_{(1,2,2)}^\pm(\Q) \to Y_1^0(N,2)(\Q)).
\end{align*}
\end{corollary}
\begin{proof}
This follows immediately from Lemma \ref{lem:Y10NM}.
\end{proof}

Here is an  analogue of Proposition \ref{prop:rat-pt-Y1NM}.

\begin{proposition}\label{prop:rat-pt-Y10NM}
Let $N, \ul{M}, \ul{\e}$ be as in Definition \ref{def:X10NM}.
\begin{enumerate}
\item 
If one of the following conditions holds,
$X_1^0(N, 2)_{\ul{M}}^{\ul{\e}}$
is isomorphic to $\P^1$ over $\Q$.
%(hence $Y_1(N)_M^\epsilon(\Q)$ contains infinitely many points).
\begin{enumerate}
\item 
$N=4$ and 
$\ul{M} \in \{(2,1,2), (1,2,2), (2,1,1), (1,2,1), (1,1,2), (1,1,1)\}$.
\item
$N=6$ and $\ul{M}
\in \{(1,1,2), (1,1,1)\}$.
\item
$N=8$ and $\ul{M}=(1,1,1)$.
\end{enumerate}
\item
If none of (a)--(c) is satisfied, then we have 
$Y_1^0(N, 2)_{\ul{M}}^{\ul{\e}}(\Q)=\emptyset$.
\end{enumerate}
\end{proposition}
\begin{proof}
(1a) 
It suffices to show the first two cases by Lemma \ref{lem:Y10NM} (3).
Since $g_4^{(3)}(u)=1+u$ defines an isomorphism $\P^1_u \to \P^1_{w_3}$,
we have 
$X_1^0(4, 2)_{(2,1,2)}^{\ul{\e}} \cong X_1^0(4, 2)_{(2,1,1)}^{\ul{\e}}$
and
$X_1^0(4, 2)_{(1,2,2)}^{\ul{\e}} \cong X_1^0(4, 2)_{(1,2,1)}^{\ul{\e}}$.
Now ${X_1^0}'(N, 2)_{(2,1,1)}^{\ul{\e}}$
is a nodal curve $4u(1+u)^2=\pm v_1^2$,
and 
${X_1^0}'(N, 2)_{(1,2,1)}^{\ul{\e}}$
is a conic $1-u^2=\pm v_2^2$ having a $\Q$-rational point,
both of which become isomorphic to $\P^1$ after normalization.

(1b) 
It suffices to show the first case by Lemma \ref{lem:Y10NM} (3).
Now $X_1^0(6, 2)_{(1,1,2)}^{\ul{\e}}$ 
is a conic defined by $(1-u)(1-3u)=\pm v_3^2$ 
having a $\Q$-rational point, which is rational.

(1c) 
This is obvious since 
${X_1^0}(8, 2)_{(1,1,1)}^{\ul{\e}} \cong {X_1^0}(8, 2)$.

(2) We divide the proof by the value of $N$.

\paragraph{{\bf Case of $N=4$}}
It suffices to show the cases $\ul{M}=(4,1,1), (2,2,1)$
by Lemma \ref{lem:Y10NM} (3).
\begin{itemize}
\item 
Suppose first $\ul{M}=(4,1,1)$ so that
${X_1^0}(4, 2)_{\ul{M}}^{\ul{\e}}$ is given by $4u(1+u)^2 = \pm v_1^4$.
It is isomorphic to an elliptic curve
$E_\pm : y^2=x^3 \pm 4x$ 
by the change of variables $(u, v_1)=(-1 \pm (y^2/4x), \pm y/2)$
($(x,y)=(0,0)$ corresponds to $(u,v_1)=(-1, 0)$).
They appear as 32.a4 and 64.a3 in \cite{LMFDB},
according to which 
$E_+(\Q)=\{(0,0), (2, \pm 4), O \}$
and 
$E_-(\Q)=\{ (0,0), (\pm 2, 0), O \}$.
We find all points of ${X_1^0}(4, 2)_{\ul{M}}^{\ul{\e}}(\Q)$ are cusps.
\item 
Next suppose $\ul{M}=(2,2,1)$ so that
${X_1^0}(4, 2)_{\ul{M}}^{\ul{\e}}$ is given by 
$4u(1+u)^2 = \pm v_1^2, \ 1-u^2 = \pm v_2^2$.
There is a finite morphism from 
${X_1^0}(4, 2)_{\ul{M}}^{\ul{\e}}$
onto an elliptic curve $E_\pm : y^2 = x^3 \pm x$
given by $(x,y)=(\mp u, (v_1v_2)/2)$.
They are labeled as 64.a4 and 32.a3 in \cite{LMFDB},
according to which 
$E_+(\Q)=\{ (0,0), O \}$
and
$E_-(\Q)=\{ (0,0), (\pm1, 0), O \}$.
We conclude all rational points on 
${X_1^0}(4, 2)_{\ul{M}}^{\ul{\e}}$ are cusps.
\end{itemize}

\paragraph{{\bf Case of $N=6$}}
It suffices to consider the cases $\ul{M}=(3,1,1), (2,1,1), (1,2,1)$.
\begin{itemize}
\item 
Suppose first $\ul{M}=(3,1,1)$. 
We may assume $\epsilon_1=+$ (as $M_1=3$ is odd),
and then
${X_1^0}(6, 2)_{\ul{M}}^{\ul{\e}}$ is given by 
$u(1-u)(1+3u)^3(1+u)^4=v_1^3$.
It is isomorphic to an elliptic curve
$E : y^2=x^3+1$
by the change of variables 
$(x,y)=(v_1/(u(1+3u)(1+u)), 1/u)$.
This appears in  LMFDB as 36.a4,
and one has $E(\Q)=\{ (-1,0), (0, \pm 1), (2, \pm 3), O\}$.
All points correspond to cusps.
%
%
%
%
%Suppose first $\ul{M}=(3,1,1)$. 
%We may assume $\epsilon_1=+$ (as $M_1=3$ is odd),
%and then
%${X_1^0}(6, 2)_{\ul{M}}^{\ul{\e}}$ is given by 
%$u(1-u)(1+3u)^3(1+u)^4=v_1^3$.
%It is isomorphic to an elliptic curve
%$E : x-x^3=y^3$
%by the change of variables 
%$(x,y)=(u, v_1/((1+3u)(1+u)))$.
%All points of $E(\Q)=\{ (0,1), (\pm 1,0), (\pm 1/3, \pm 2/3), O \}$
%correspond to cusps.
\item 
Suppose $\ul{M}=(2,1,1)$ so that
${X_1^0}(6, 2)_{\ul{M}}^{\ul{\e}}$ is given by 
$u(1-u)(1+3u)^3(1+u)^4=\pm v_1^2$.
It is isomorphic to an elliptic curve
$E_\pm : y^2 = x^3 \mp x^2 - 4x \pm 4$
by $(x,y)=(\pm(1-3u), 3v_1/((1+3u)(1+u)^2))$.
Its label in LMFDB is 24.a4 or 48.a4 according to $\epsilon_1=\pm$,
and their rational points are given by 
\begin{align*}
&E_+(\Q)=\{ (\pm 2,0), (1, 0),  (0,\pm 2), (4, \pm 6), O \},
\\
&E_-(\Q)=\{ (\pm 2, 0), (-1,0), O \}
\end{align*}
respectively.
Again, all points correspond to cusps.
%
%
%Suppose $\ul{M}=(2,1,1)$ so that
%${X_1^0}(6, 2)_{\ul{M}}^{\ul{\e}}$ is given by 
%$u(1-u)(1+3u)^3(1+u)^4=\pm v_1^2$.
%It is isomorphic to an elliptic curve
%$E_\pm : x(1-x)(1+3x)=\pm y^2$
%by $(x,y)=(u, v_1/((1+3u)(1+u)^2))$.
%Again, all points of
%\begin{align*}
%&E_+(\Q)=\{ (0,0), (1,0), (-1, \pm 2), (-1/3, 0), (1/3, \pm 2/3), O \},
%\\
%&E_-(\Q)=\{ (0,0), (1,0), (-1/3, 0), O \}
%\end{align*}
%correspond to cusps.
\item 
Finally, suppose $\ul{M}=(1,2,1)$ so that
${X_1^0}(6, 2)_{\ul{M}}^{\ul{\e}}$ is given by 
$(1-u^2)(1-9u^2)=\pm v_2^2$.
It is isomorphic to the same elliptic curve $E_\pm$ as 
the one discussed in the previous case $\ul{M}=(2,1,1)$
(i.e. 24.a4 or 48.a4)
by $(x,y)=(\pm 4u/(u-1), 2v_2/(u-1)^2)$.
This settles the case $N=6$.
%
%
%
%Finally, suppose $\ul{M}=(1,2,1)$ so that
%${X_1^0}(6, 2)_{\ul{M}}^{\ul{\e}}$ is given by 
%$(1-u^2)(1-9u^2)=\pm v_2^2$.
%We find 
%${X_1^0}(6, 2)_{\ul{M}}^{\ul{\e}}(\Q)=\{ (\pm 1,0), (\pm 1/3, 0), (0, \pm 1), \pm \infty \}$
%if $\e_2=+$,
%and
%${X_1^0}(6, 2)_{\ul{M}}^{\ul{\e}}(\Q)=\{ (\pm 1,0), (\pm 1/3, 0)\}$
%if $\e_2=-$
%(in $(u, v_2)$-coordinates).
%Again, these points correspond to cusps.
\end{itemize}

\paragraph{{\bf Case of $N=8$}}
It suffices to consider the cases $\ul{M}=(2,1,1), (1,2,1), (1,1,2)$.
\begin{itemize}
\item 
Suppose first $\ul{M}=(2,1,1)$.
By Lemma \ref{lem:Y10NM} (2), 
$(E, P, Q) \in Y_1^0(8, 2)(\Q)$ 
is in the image from ${Y_1^0}(8, 2)_{\ul{M}}^{\pm}(\Q)$
if and only if $\gen{P, 4P}=0$.
This implies $(E, 2P, Q) \in Y_1^0(4, 2)(\Q)$
is in the image from ${Y_1^0}(4, 2)_{(4,1,1)}^{\pm}(\Q)$,
but we have already shown the latter set is empty.
\item 
Next suppose $\ul{M}=(1,2,1)$ so that
${X_1^0}(8, 2)_{\ul{M}}^{\ul{\e}}$ is a genus one curve
defined by $1-6u^2+u^4 = \pm v_2^2$.
There is a finite morphism from 
${X_1^0}(8, 2)_{\ul{M}}^{\ul{\e}}$ onto an elliptic curve
$E : y^2 = x^3-x$
(the one used in the case $N=4$ and $\ul{M}=(2,2,1)$, LMFDB label 32.a3)
given by 
$(x,y)=((v_2+1-u^2)/(2u^2), (v_2+1-3u^2)/(2u^3))$
or
$(x,y)=((v_2+2u)/(u-1)^2, (u+1)(v_2-u^2+4u-1)/(u-1)^3)$ 
according to $\epsilon_2=+$ or $-$.
%(Actually a bijection, with the inverse $(u,v)=((x-1)/y, (x^2-2x+1)/(x^2+x))$.)
Since $E(\Q)=\{ (0,0), (\pm1, 0), O \}$,
we conclude all rational points are cuspidal.
%
%
%
%Next suppose $\ul{M}=(1,2,1)$ so that
%${X_1^0}(8, 2)_{\ul{M}}^{\ul{\e}}$ is a genus one curve
%defined by $1-6u^2+u^4 = \pm v_2^2$.
%We have 
%${X_1^0}(8, 2)_{\ul{M}}^{\ul{\e}}= \{(0,\pm 1), \pm \infty \}$ if $\e_2=+$,
%and
%${X_1^0}(8, 2)_{\ul{M}}^{\ul{\e}}= \{(\pm 1,\pm 2), \pm \infty \}$ if $\e_2=-$
%(in $(u, v_2)$-coordinates),
%all of which are cuspidal.
\item 
Finally, suppose $\ul{M}=(1,1,2)$ so that
${X_1^0}(8, 2)_{\ul{M}}^{\ul{\e}}$ is a hyperelliptic curve defined by
$(1-u^4)(1+2u-u^2)=\pm v_3^2$.
In this case, we have
${X_1^0}(8, 2)_{\ul{M}}^{\ul{\e}}(\Q)
=\{ (\pm 1,0), (0, \pm 1), \pm \infty \}$
if $\e_3=+$,
and
${X_1^0}(8, 2)_{\ul{M}}^{\ul{\e}}(\Q)=\{ (\pm 1,0) \}$
if $\e_3=-$
(in $(u, v_3)$-coordinates).
This can be verified for instance by a Magma code
\footnote{
Available online at
\url{https://magma.maths.usyd.edu.au/calc/}
}\begin{verbatim}
P<x>:=PolynomialRing(Rationals());
C:=HyperellipticCurve((1-x^4)*(1+2*x-x^2));
PointsGenus2(C);
\end{verbatim}
(and by a similar code with change of sign).
These points are cuspidal.
\qedhere
\end{itemize}
\end{proof}

\section{End of the proof}\label{sect:ref-mazur2}
We are now in position to
complete the proof of Theorem \ref{theorem: intrinsic torsion}.

\subsection{Proof of non-existence}
We first consider the case $E(\Q)_\Tor$ is cyclic.
Let $E$ be an elliptic curve over $\Q$,
$A=\Z/N\Z$ with $N \in \{ 6,7,8,9,10,12 \}$,
and $\alpha : E(\Q)_\Tor \cong A$ an isomorphism.
Take a generator $P \in E(\Q)_\Tor$ and a positive divisor $M|N$
such that $M$ does not appear as the order of $B$ in the list of 
Theorem \ref{theorem: intrinsic torsion}.
If $(N/M)P \in E(\Q)_\Tor^\is$ would hold,
Lemma \ref{lem:Y1NM} implies 
$(E, P) \in \Im(Y_1(N)_M^\pm(\Q) \to Y_1(N)(\Q))$,
but we have $Y_1(N)_M^\pm(\Q)=\emptyset$ by Proposition \ref{prop:rat-pt-Y1NM}.
This proves that there is no $E$ such that $(E(\Q)_\Tor, E(\Q)_\Tor^\is)=(\gen{P}, \gen{(N/M)P})$.

The proofs of the cases
$A=\Z/N\Z \times \Z/2\Z$ for $N = 4, 6, 8$ are similar.
As a sample, we detail the case $N=6$.
Let $E$ an elliptic curve over $\Q$,
and let $\alpha : E(\Q)_\Tor \cong A:=\Z/6\Z \times \Z/2\Z$ be an isomorphism.
If $E(\Q)_\Tor^\is$ were non-trivial,
there should exist $P, Q \in E(\Q)_\Tor$ such that
$(E, P, Q) \in Y_1^0(6, 2)(\Q)$ and 
$\{ 2P, 3P, Q \} \cap E(\Q)_\Tor^\is \not= \emptyset$.
Since we have
$Y_1^0(6, 2)_{(3,1,1)}^\pm(\Q)=
Y_1^0(6, 2)_{(2,1,2)}^\pm(\Q)=
Y_1^0(6, 2)_{(1,2,2)}^\pm(\Q)=\emptyset$
by Proposition \ref{prop:rat-pt-Y10NM},
this would contradict Lemma \ref{lem:Y10NM} (2).
The cases $N=4, 8$ are similar and skipped.

The reason for the non-existence in the case $A=\Z/2\Z \times \Z/2\Z$ 
is quite different.
Indeed, there do exist infinitely many
$(E, P, Q) \in Y_1^0(2,2)(\Q)$ with 
$\gen{P, P}=\gen{Q, Q}=\gen{P, Q}=0$.
However, all of them come from $Y_1^0(4,2)(\Q)$,
as is shown in the following proposition,
which completes the proof of the non-existence part of 
Theorem \ref{theorem: intrinsic torsion}.

\begin{proposition}
Let $(E, P, Q) \in Y_1^0(2,2)(\Q)$ 
and suppose that $P, Q \in E(\Q)_\Tor^\is$.
Then $E(\Q)_\Tor$ contains an element of order four.
Consequently, there is no elliptic curve $E$ over $\Q$ such that
  $E(\Q)_\Tor=E(\Q)_\Tor^\is\simeq\Z/2\Z\times\Z/2\Z$.
\end{proposition}
\begin{proof}
We may assume $E$ is given by 
$E_{2,a,u}':y^2=x(x-a)(x-au)$ from \eqref{eq: E2ua}
for some $a, u \in \Q$ with $au(1-u) \not=0$,
and $P=(0,0), Q=(a,0)$.
By Lemma \ref{lemma: pairing E2u}, we have 
$\gen{P,P}=u\otimes(1/2)$,
  $\gen{Q,Q}=(1-u)\otimes(1/2)$ and $\gen{P,Q}=a\otimes(1/2)$.
The assumption $P, Q \in E(\Q)_\Tor^\is$ implies that
$u=\pm b^2$, $1-u=\pm c^2$, and $a=\pm d^2$ for some $b,c,d\in(\Q^\times)^2$. 
Recall from Example 6.2 of \cite{Husemoller} that 
an elliptic curve $y^2=x(x-A)(x-B) \ (A, B \in \Q)$ has a
  rational torsion point $R$ of order $4$ such that $2R=(0,0)$ if and only if
  $A$ and $B$ are squares in $\Q^\times$.
In the case
  $(u,1-u,a)=(b^2,\pm c^2,d^2)$, the elliptic curve $E_{2,a,u}'$ is
  $y^2=x(x-d^2)(x-b^2d^2)$, and hence it has
  a torsion point of order $4$.
In the case
  $(u,1-u,a)=(-b^2,c^2,d^2)$, we make a change a variable
  $(x,y)\mapsto(x-b^2d^2,y)$ in $E_{2,a,u}'$ and find that $E_{2,a,u}'$
  is isomorphic to $y^2=x(x-b^2d^2)(x-c^2d^2)$. Again, this elliptic
  curve has a rational point of order $4$. When
  $(u,1-u,a)=(-b^2,c^2,-d^2)$, we make a change of variables
  $(x,y)\mapsto(x-d^2,y)$ and find that $E_{2,a,u}'$ is isomorphic to
  $y^2=x(x-d^2)(x-c^2d^2)$, which has a rational point of order $4$ as well.
\end{proof}

\subsection{Proof of existence}
First let us consider the cases
$A=\Z/N\Z$ with $N =1, 7, 9, 10, 12$ or $A=\Z/N\Z \times \Z/2\Z$ with $N =6, 8$.
We have shown that 
$E(\Q)_\Tor^\is$ must be trivial if $E(\Q)_\Tor \cong A$.
Since there are infinitely many $E$ with $E(\Q)_\Tor \cong A$
by Mazur's theorem, the existence part of 
Theorem \ref{theorem: intrinsic torsion} follows in these cases.

To proceed, 
we use the following fact about the \emph{thin sets}
(see \cite[Proposition 3.4.2]{serre}).

\begin{proposition}\label{fact}
If $f : \P^1 \to \P^1$ is a finite morphism of degree $d$,
then the number of points of $f(\P^1(\Q)) \subset\P^1(\Q)$ with height $\le h$
is asymptotically $\sim h^{2/r}$ as $h \to \infty$.
\end{proposition}

For convenience, 
we describe the degrees of relevant morphisms.
\[
\xymatrix{
& X_1(12) \ar[r]^8 \ar[d]^4 & X_1(4) \ar[d]^2 & X_1(8) \ar[l]_4
\\
& X_1(6) \ar[r]^4 \ar[d]^3 & X_1(2) \ar[d]^3 & X_1(10) \ar[l]_{12} \ar[d]^3
\\
X_1(9) \ar[r]^9 & X_1(3) \ar[r]^4 & X_1(1) & X_1(5) \ar[l]_{12}
\\
X_1^0(6,2) \ar[r]^4 \ar[d]^2 & X_1^0(2,2) \ar[d]^2 & 
X_1^0(4,2) \ar[d]^2 \ar[l]_2 & X_1^0(8,2) \ar[d]^2 \ar[l]_4
\\
X_1(6) \ar[r]^4 & X_1(2)  & X_1(4) \ar[l]_2 & X_1(8) \ar[l]_4
}
\]
Recall also that the degree of $X_1(N)_M^\epsilon \to X_1(N)$ is $M$,
and
that of $X_1^0(N,2)_{\ul{M}}^{\ul{\e}} \to X_1^0(N,2)$ is $M_1M_2M_3$.
Our proof goes from a bigger groups to smaller.

\paragraph{{\bf Case of $A=\Z/4\Z \times \Z/2\Z$}}
Let us introduce three subsets of $Y_1^0(4,2)(\Q)$:
\begin{align*}
J_1:=& \Im(Y_1^0(8,2)(\Q) \to Y_1^0(4,2)(\Q)),
\\
J_2:=& \Im(Y_1^0(4,2)^\pm_{(2,1,1)}(\Q) \to Y_1^0(4,2)(\Q)),
\\
J_3:=& \Im(Y_1^0(4,2)^\pm_{(1,2,2)}(\Q) \to Y_1^0(4,2)(\Q)).
\end{align*}
Note that $J_1 \cap J_3 = \emptyset$.
Indeed, 
since $Y_1^0(8, 2)_{(1,2,1)}^\pm(\Q)=\emptyset$ by Proposition \ref{prop:rat-pt-Y10NM},
we have $\gen{Q,Q} \not= 0$ for any $(E, P', Q) \in Y_1^0(8, 2)(\Q)$,
which implies the claim.

For $(E, P, Q) \in Y_1^0(4, 2)(\Q)$, we have
$E(\Q)_\Tor=\gen{P, Q} \Leftrightarrow (E, P, Q) \not\in J_1$.
Suppose this holds.
Then we know that $|E(\Q)_\Tor^\is| \le 2$ (by the non-existence part).
Therefore we have the equivalences
\begin{itemize}
\item $E(\Q)_\Tor^\is=\gen{2P} \Leftrightarrow (E, P, Q) \in J_2$, 
\item $E(\Q)_\Tor^\is=\gen{Q} \Leftrightarrow (E, P, Q) \in J_3$,
\item $E(\Q)_\Tor^\is=0 \Leftrightarrow (E, P, Q) \not\in J_2 \cup J_3$.
\end{itemize}
We conclude the existence part of Theorem \ref{theorem: intrinsic torsion}
in this case by Proposition \ref{fact}.

%Let $(E, P, Q) \in Y_1^0(4, 2)(\Q)$,
%and suppose first  $E(\Q)_\Tor^\is$ is non-trivial.
%Then we must have $E(\Q)_\Tor \cong A$
%(because we have shown $E(\Q)_\Tor \cong \Z/8\Z \times \Z/2\Z$
%implies $E(\Q)_\Tor^\is=0$),
%and then $|E(\Q)_\Tor^\is|=2$
%(because we have shown $|E(\Q)_\Tor^\is| \not= 4, 8$).
%Thus we have
%\begin{itemize}
%\item 
%$E(\Q)_\Tor^\is=\gen{2P} \Leftrightarrow
%(E, P, Q) \in \Im(Y_1^0(4,2)_{(2,1,1)}^\pm(\Q) \to Y_1^0(4,2)(\Q))$,
%\item 
%$E(\Q)_\Tor^\is=\gen{Q} \Leftrightarrow
%(E, P, Q) \in \Im(Y_1^0(4,2)_{(1,2,2)}^\pm(\Q) \to Y_1^0(4,2)(\Q))$.
%\end{itemize}
%Since the components of 
%$Y_1^0(N)_{(2,1,1)}^\pm$ and $Y_1^0(N)_{(1,2,2)}^\pm$
%are isomorphic to $\P^1$ by Proposition \ref{prop:rat-pt-Y10NM},
%we conclude by Proposition \ref{fact} 
%that there are infinitely many such $(E, P, Q)$ in both cases.
%Finally, 
%we have $E(\Q)_\Tor^\is=0$ precisely when
%neither of the above conditions are met
%and $(E, P, Q) \not\in \Im(Y_1^0(8,2)(\Q) \to Y_1^0(4,2)(\Q))$.
%Again we conclude by Proposition \ref{fact} 
%that there are infinitely many such $(E, P, Q)$.

\paragraph{{\bf Case of $A=\Z/2\Z \times \Z/2\Z$}}
We need a special treatment for this case.
Consider the following subsets of $Y_1(1)(\Q)$.
\begin{align*}
J:=& \Im(Y_1^0(2,2)(\Q) \to Y_1(1)(\Q)),
\\
J_1:=& \Im(Y_1^0(4,2)(\Q) \sqcup Y_1^0(6,2)(\Q)  \to Y_1(1)(\Q)),
\\
J_2:=& \Im(Y_1^0(4,2)_{(2,1,1)}^\pm(\Q) \sqcup Y_1^0(4,2)_{(1,2,2)}^\pm(\Q)  \to Y_1(1)(\Q)),
\\
J_3:=& \{ j(E) \in Y_1(1)(\Q) \mid (E, P, Q) \in Y_1^0(2,2)(\Q), \ E(\Q)_\Tor^\is \not=0 \}.
\end{align*}
By Corollary \ref{cor:criterion2} and
the (proved) non-existence result, we obtain
\begin{itemize}
\item 
$J \setminus (J_1 \cup J_3)$ is precisely the set of the $j$-invariants of
elliptic curves $E$ over $\Q$ such that
$E(\Q)_\Tor \cong A$ and $E(\Q)_\Tor^\is=0$.
\item 
$J_3 \setminus J_2$ is precisely the set of the $j$-invariants of
elliptic curves $E$ over $\Q$ such that
$E(\Q)_\Tor \cong A$ and $|E(\Q)_\Tor^\is|=2$.
\end{itemize}
Furthermore,  we claim that there is a morphism 
$\wt{j} : \P^1 \to \P^1$ of degree $12$ such that
$J_3=\wt{j}(\P^1(\Q) \setminus \{ 0, \pm 1, \infty \})$.
By Proposition \ref{fact},
this will complete the proof of this case.

To show the claim,
we make  use of the elliptic curve 
$E_{2,a,u}':y^2=x(x-a)(x-au)$ from \eqref{eq: E2ua}
for $a, u \in \Q$ with $au(1-u) \not=0$ again.
Observe that its $j$-invariant is given by 
\[
j(u)=256\frac{(1-u+u^2)^3}{u^2(1-u)^2}.
\]
It suffices to show $J_3=\{ j(\pm v^2) \mid v \in \Q, v \not= 0, \pm 1 \}$.
Let $(E,P,Q) \in Y_1^0(2,2)(\Q)$ and suppose $P \in E(\Q)_\Tor^\is$.
There is an isomorphism $E \cong E_{2,a,u}'$
for some $a, u \in \Q$ with $au(1-u) \not=0$, 
by which $P$ corresponds to $(0,0)$.
Since $P \in E(\Q)_\Tor^\is$,
we have $u=\pm v^2$ for some $v \in \Q \setminus \{ 0, \pm 1 \}$
by Lemma \ref{lemma: pairing E2u}.
Hence the $j(u)=j(\pm v^2)$ belongs to $J_3$. 
The converse is seen by taking $u=\pm v^2$ (and $a=1$).
This completes the proof of the claim and hence the existence part of this case.

\paragraph{{\bf Cyclic case}}
In view of Proposition \ref{fact},
the cases $A=\Z/N\Z, N=4,5,6,8$
now follow from  Corollary \ref{cor:criterion1}.
The remaining cases $N=2, 3$
are proved by the same way as $A=\Z/2\Z \times \Z/2\Z$,
using the formulas
\begin{equation*}\label{eq:j-inv-E23}
 j(E_{2,t,a})=
64 \frac{(t+4)^3}{t^2},
%\begin{cases}
%1024 \frac{(t+1)^2(t+4)}{a^4t^2}, & (t \not= -1),
%\\
%1728 & (t=-1),
%\end{cases}
\qquad
\begin{cases}
 j(E_{3,t})=27 \frac{(t+1)(t+9)^3}{t^3} & (t \not= -1),
\\
j(E_{3,-1, a})=0
\end{cases}
\end{equation*}
for the $j$-invariants of
the elliptic curves from \eqref{eq: E2ta}, \eqref{eq: E3t} and \eqref{eq: E3a}.

This completes the proof of Theorem \ref{theorem: intrinsic torsion}.
\qed

\section{Relation with the reduction type}\label{sect:high-dim}

\subsection{Second construction of the pairing}
In this subsection,
we let $X$ be a smooth projective geometrically connected variety 
over a field $k$ of dimension $d$.
Let $\CH^i(X)$ denote the Chow group of codimension $i$ cycles on $X$,
and put $A_0(X):=\ker(\deg : \CH^d(X) \to \Z)$.
We shall construct a biadditive pairing 
\begin{equation}\label{eq:pairing2}
\langle \cdot, \cdot \rangle : 
\CH^1(X)_\Tor \times A_0(X) \to k^\times \otimes \Q/\Z,
\end{equation}
and show that it agrees with
\eqref{eq:pairing} when $d=1$
(under the identification $\CH^1(X)=\Pic(X)$).
This construction is a key ingredient in the proof of
Proposition \ref{prop:two-def-same} below.

We shall use Bloch's higher Chow group $\CH^r(X, i, R)$ 
with coefficient in a commutative ring $R$.
We abbreviate $\CH^r(X, i):=\CH^r(X, i, \Z)$
so that $\CH^r(X)=\CH^r(X, 0)$. % and $\CH^r(X)/n=\CH^r(X, 0; \Z/n\Z)$.
We fix $n>0$
and define
a biadditive map $\langle \cdot, \cdot \rangle $ as the composition of
\begin{equation}
\begin{split}
\langle \cdot, \cdot \rangle 
&: 
\CH^1(X, 1; \Z/n\Z) \times \CH^d(X, 0; \Z/n\Z)
\\
&\to \CH^{d+1}(X, 1; \Z/n\Z)
\overset{s_*}{\to} \CH^1(\Spec k, 1; \Z/n\Z)
\cong k^\times \otimes \Z/n\Z,
\end{split}
\end{equation}
where the first map is given by
the multiplicative structure of the higher Chow groups,
and $s_*$ is the push-forward along the structure map
$s : X \to \Spec k$.
We have an exact sequence and an isomorphism
\begin{align*}
&0 \to 
\CH^1(X, 1)/n \overset{j}{\to} \CH^1(X, 1; \Z/n\Z) \to 
\CH^1(X)[n] \to 0,
\\
&\CH^d(X)/n \cong \CH^d(X, 0; \Z/n\Z).
&
\end{align*}
We claim that
$\langle j(a), b \rangle = 0$ for
$a \in \CH^1(X, 1)$ and $b \in A_0(X) := \ker(\deg : \CH^d(X) \to \Z)$.
To show this, 
we first recall that $\CH^1(V, 1) \cong \sO(V)^\times$
for any smooth variety $V$.
By our assumption on $X$,
we obtain isomorphisms
\[ k^\times \cong \CH^1(\Spec k, 1) \overset{s^*}{\longrightarrow} \CH^1(X, 1). \]
Thus it suffices to show that
$\langle j(s^*(a)), b \rangle_X = a^{\deg(b)}$
for any $a \in \CH^1(\Spec k, 1)$ and $b \in \CH^d(X)$.
We may also suppose $b=[x]$ is the class of a closed point $x \in X$.
Write $i_x : x=\Spec k(x) \to X$ for the closed immersion.
Now the assertion follows from the commutative diagram
\[
\xymatrix{
\CH^1(X, 1) \ar[r]^{- \cdot [x]} \ar[rd]^{i_x^*} &
\CH^{d+1}(X, 1) \ar[r]^{s_*} &
\CH^1(\Spec k, 1)
\\
\CH^1(\Spec k, 1) \ar[r]_-{(s \circ i_x)^*} \ar[u]^{s^*}_\cong
\ar@/_17mm/[rru]_{[k(x):k].} 
&
\CH^1(x, 1) \ar[u]^{i_{x*}} \ar[ur]_{(s\circ i_x)_*} &
}
\]
Therefore we obtain an induced  biadditive pairing
\begin{equation}\label{eq:pairing-CH-n}
\CH^1(X)[n] \times A_0(X)/n \to k^\times/n,
\end{equation}
and, by taking colimit over $n$, \eqref{eq:pairing2} as well.

%\begin{problem}
%Can this be generalized to a pairing
%$\CH^r(X)_\Tor \times \CH_\hom^{d-r+1}(X)  \to k^\times \otimes \Q/\Z$?
%%(Here $\CH_\hom^*(X):=\ker(\CH^*(X) \to H^{2*}_\et(X_{\ol{k}}, \wh{Z}(*)))$.
%(The above method does not work, 
%since $0=\CH^r(\Spec k, 1) \to \CH^r(X, 1)$ may not be bijective for $r>1$.)
%\end{problem}

\begin{proposition}\label{prop:comparison}
Suppose $d=1$. 
Then the two pairings \eqref{eq:pairing} and \eqref{eq:pairing-CH-n}
are equal to each other,
under the canonical identification $\CH^1(X) \cong \Pic(X)$.
\end{proposition}
\begin{proof}
Take $a \in \Pic(X)_\Tor \ (a \not= 0)$ and $b \in \Pic^0(X)=A_0(X)$.
We choose $D, E \in \Div(X)$ such that
$a=[D], b=[E]$ and $|D| \cap |E|=\emptyset$.
We also take $n \in \Z_{>0}$ such that $na=0$
and $f \in k(X)^\times$ such that $\div(f)=nD$.
Writing $E=\sum_j e_j Q_j$,
we have $\langle a, b \rangle = \prod_j N_{k(Q_j)/k}(f(Q_j))^{e_j} \otimes (1/n)$
under \eqref{eq:pairing}.
We shall prove the same formula for \eqref{eq:pairing-CH-n}.

Let $\ol{\Gamma}_f \subset X \times \P^1$
be the graph of the finite morphism $X \to \P^1$ defined by $f$,
and $\Gamma_f := \ol{\Gamma}_f \cap (X \times \cube)$,
where $\cube:=\P^1 \setminus \{ 1 \}$.
For $\epsilon \in \{ 0, \infty \} \in \cube$, 
denote by $i_\epsilon : X \to X \times \cube$
the corresponding closed immersion.
Then $\Gamma_f$ intersects properly with faces,
and its boundary is given by 
$i_0^*(\Gamma_f)-i_\infty^*(\Gamma_f)=\div(f)=nD$.
Hence $\Gamma_f$ defines an element of $\CH^1(X, 1; \Z/n\Z)$
whose image in $\CH^1(X)[n]=\Pic(X)[n]$ agrees with $[D]=a$.
The product $[\Gamma_f] \cdot b \in \CH^2(X, 1; \Z/n\Z)$
is represented by
the intersection product $\Gamma_f \cdot (E \times \cube)$ on $X \times \cube$,
and its push-forward onto $\Spec k$
by $\sum_j e_j f(Q_j) \in \Div(\cube)$.
By the definition of the isomorphism
$\CH^1(\Spec k, 1; \Z/n\Z) \cong k^\times \otimes \Z/n\Z$ 
(see \cite[p. 183]{Totaro}),
this element corresponds to 
$\prod_j N_{k(Q_j)/k}(f(Q_j))^{e_j} \otimes (1/n) \in k^\times \otimes \Z/n\Z$.
This completes the proof.
\end{proof}

\begin{definition}
We define the intrinsic subgroups of $X$ by
\begin{align*}
\CH^1(X)_\Tor^\is 
:= \{ a \in \CH^1(X)_\Tor \mid \langle a, b \rangle = 0
\text{ for all } b \in A_0(X)_\Tor \},
\\
A_0(X)_\Tor^\is 
:= \{ b \in A_0(X)_\Tor \mid \langle a, b \rangle = 0
\text{ for all } a \in \CH^1(X)_\Tor \}.
\end{align*}
\end{definition}

\subsection{Good reduction}\label{sect:goodred}
As an application of the second construction of the pairing, 
we prove that there is some restriction 
on the possible values of \eqref{eq:pairing2}
if $X$ has good reduction with respect to a discrete valuation.

%We have a split exact sequence $0 \to O^\times/n \to k^\times/n \to \Z/n\Z \to 0$
%for any $n>0$.

\begin{proposition}\label{prop:two-def-same}
Let $v$ be a discrete valuation on $k$, and $O$ its valuation ring.
Let $\sX \to \Spec O$ be a smooth proper morphism with geometrically connected fibers,
and denote by $X$ the generic fiber.
Then we have
\[ \Im(\langle \cdot, \cdot \rangle : 
\CH^1(X)_\Tor \times A_0(X) \to k^\times \otimes \Q/\Z)
\subset
\Im(O^\times \otimes \Q/\Z \to k^\times \otimes \Q/\Z).
\]
\end{proposition}
\begin{proof}
We fix $n>0$ and prove 
\[ \Im(\langle \cdot, \cdot \rangle : \CH^1(X)[n] \times A_0(X)/n \to k^\times/n)
\subset
\Im(O^\times/n \to k^\times/n).
\]
We use Bloch's higher Chow group 
for schemes over a discrete valuation ring,
for which we refer to \cite{L}.
Take $a \in \CH^1(X)[n]$ and $b \in A_0(X)/n$.
We shall show $\langle a, b \rangle \in O^\times/n$.
Let $d:=\dim X$.
Since $\CH^1(\sX) \overset{\cong}{\to} \CH^1(X)$ is an isomorphism,
the map $\phi$ in the diagram
\[
\xymatrix{
0 \ar[r] &
O^\times/n \ar[r] \ar[d] &
\CH^1(\sX, 1; \Z/n\Z) \ar[r] \ar[d] \ar[rd]^\phi&
\CH^1(\sX)[n] \ar[r] \ar[d]^\cong &
0 \\
0 \ar[r] &
k^\times/n \ar[r]  &
\CH^1(X, 1; \Z/n\Z) \ar[r]  &
\CH^1(X)[n] \ar[r] &
0
}
\]
is surjective.
Thus there is $\wt{a} \in \CH^1(\sX, 1; \Z/n\Z)$
such that $a=\phi(\wt{a})$.
Let $\psi$ be the composition of
surjective canonical maps
\begin{align*}
\psi : 
\CH^d(\sX, 1, \Z/n\Z) \twoheadrightarrow 
\CH^d(X, 1, \Z/n\Z) \twoheadrightarrow \CH^d(X)/n.
\end{align*}
Then  there is $\wt{b} \in \CH^d(\sX, 1; \Z/n\Z)$
such that $b=\psi(\wt{b})$.
By the commutative diagram
\[
\xymatrix{
\CH^1(\sX, 1; \Z/n\Z) \times \CH^d(\sX)/n \ar[d] \ar[r]
\ar@/_30mm/[ddd]_{\gamma} &
\CH^1(X, 1; \Z/n\Z) \times \CH^d(X)/n \ar[d] 
\\
\CH^{d+1}(\sX, 1; \Z/n\Z)\ar[d] \ar[r]&
\CH^{d+1}(X, 1; \Z/n\Z)\ar[d]
\\
\CH^1(\Spec O, 1; \Z/n\Z)  \ar[d]_\cong \ar[r]&
\CH^1(\Spec k, 1; \Z/n\Z) \ar[d]^\cong 
\\
O^\times/n \ar[r]^i &
k^\times/n,
}
\]
we find that 
$\langle a, b \rangle = i(\gamma(\wt{a}, \wt{b})) \in O^\times/n$.
We are done.
\end{proof}

\begin{example}
Let $p$ be a prime and $k=\Q_p$.
Then $\CH^1(X)_\Tor^\is$ (resp. $A_0(X)_\Tor^\is$) contains all 
elements of $\CH^1(X)_\Tor$ (resp. $A_0(X)_\Tor$) having order prime to $p$.
This follows from Proposition \ref{prop:two-def-same},
since $\Z_p^\times \otimes \Q/\Z \cong \Q_p/\Z_p$ is $p$-primary torsion.
In particular, 
we have $E(\Q_p)^\is_\Tor=E(\Q_p)_\Tor$
if $p \ge 5$ and if $E$ is an elliptic curve with good supersingular reduction
(so that $|E(\Q_p)_\Tor|$ is a divisor of $p+1$).
%(see \cite[Appendix]{Katz})
\end{example}

The following corollary is an immediate consequence of Proposition \ref{prop:two-def-same}.

\begin{corollary}
Suppose that $k$ is a number field
and that $X$ has good reduction outside a finite set $S$ of 
places of $k$ containing all archimedean places.
Let $O_{k, S}$ be the ring of $S$-integers.
Then we have
for any $D, E \in \Div_t(X)$
\[ \langle D, E \rangle \in 
\Im(O_{k, S}^\times \otimes \Q/\Z \to k^\times \otimes \Q/\Z).
\]
\end{corollary}

\begin{example}
Let $E$ be an elliptic curve over $\Q$
and let $p_1, \dots, p_n$ be all the primes at which $E$ has bad reduction.
Then we have
for any $P, Q \in E(\Q)_\Tor$
\[ \langle P, Q \rangle \in 
\left(
\{ \pm 1 \} \cdot 
\bigoplus_{i=1}^n p_i^\Z \right) \otimes \Q/\Z \subset \Q^\times \otimes \Q/\Z.
\]
\end{example}

\subsection{Tate curves}
In this subsection,
we study the intrinsic subgroup of 
an elliptic curve with split multiplicative reduction.

Assume that $k$ is complete with respect to a normalized discrete valuation 
$v : k^\times \to \Z$.
We take $q \in k^\times$ such that $n:=v(q)>0$,
and let $E_q:=\G_m/q^\Z$ be the Tate curve with parameter $q$.
We write $[a] \in E_q(k)$ 
for the point corresponding to the class of $a \in k^\times$
under the canonical isomorphism $E_q(k) \cong k^\times/q^\Z$.
We have exact sequences
\begin{equation}\label{eq:ex-TateCvTor}
0 \to O_k^\times \overset{[ \cdot ]}{\to} E_q(k) \overset{v}{\to} \Z/n\Z \to 0,
\quad \text{and} \quad
 0 \to \mu(k) \overset{[ \cdot ]}{\to} E_q(k)_\Tor \overset{v}{\to} \Z/n\Z, 
\end{equation}
where $O_k$ is the valuation ring of $k$.
We take $m>0$ and $a \in k^\times$ such that $[a] \in E_q(k)[m]$. 
We then have $a^m = q^{\wt{s}_m(a)}$ for some $\wt{s}_m(a) \in \Z$,
which is necessarily given by $\wt{s}_m(a)=v(a)m/n$.
Thus 
we obtain a well-defined map 
\[ s_m : E_q(k)[m] \to \Z/m\Z, \qquad s_m([a])=\wt{s}_m(a) \bmod m \in \Z/m \Z. \]
This fits in an exact sequence
\begin{equation}\label{eq:ex-TateCv}
0 \to \mu_m(k) 
\overset{[ \cdot ]}{\to}  E_q(k)[m] 
\overset{s_m}{\to} \Z/m\Z 
\overset{\delta_m}{\to} k^\times/(k^\times)^m,
\end{equation}
where $\delta_m$ is given by $\delta_m(1)=q \cdot (k^\times)^m$.

\begin{proposition}\label{prop:TateCv}
\begin{enumerate}
\item 
Let $a, b \in k^\times$ and $m>0$.
If $[a],  [b] \in E_q(k)[m]$, we have
\[ \langle [a], [b] \rangle 
= a^{-s_m([b])} \otimes \frac{1}{m}
= b^{-s_m([a])} \otimes \frac{1}{m}
= q^{-s_m([a])s_m([b])} \otimes \frac{1}{m^2}.
\]
\item
We have $E_q(k)_\Tor^\is = \{ [\zeta] \mid \zeta \in \mu(k) \}$.
\end{enumerate}
\end{proposition}
\begin{proof}
(1) Put $s:=\wt{s}_m(a), s':=\wt{s}_m(b) \in \Z$.
The last two equalities easily follow from $a^m=q^{s}, \ b^m=q^{s'}$.
We show the first equality.
It is obvious if $[a]=[1]$, 
hence we assume $[a] \not= [1]$ in what follows.
We use the $p$-adic theta function 
\[
\theta(u) := (1-u) \prod_{i \ge 1} \frac{(1-q^iu)(1-q^iu^{-1})}{(1-q^i)}^2
%\in k\llangle \frac{u}{q}, \frac{q}{u} \rrangle,
\]
which enjoys the fundamental equality (see \cite[V, Proposition 3.2]{sil}):
\begin{equation}\label{eq:theta-q}
-u \theta(qu)=\theta(u).
\end{equation}
It follows from this formula (and $a^m=q^s$) that 
\[
f(u) := 
\frac{\theta(u/a)^m}
{\theta(u)^{m-1} \theta(u/q^s)}
\]
satisfies $f(qu)=f(u)$.
Hence we have $f(u) \in k(E_q)^\times$ and $\div(f)=m([a] - [1])$.
We claim
\begin{align}
\label{eq:llc-theta}
\lc_{[b]}(f, m)
=q^{s(s-1)/2} b^{-s} \otimes \frac{1}{m}
\end{align}
for any $b \in k^\times$.
Indeed, 
when $[b]=[1]$
this is seen as
\begin{align*}
\lc_{[1]}(f, m)
&= \left( \frac{f(u)}{\theta(u)^m} \right)|_{u=1} \otimes \frac{1}{m}
= \theta(1/a)^m \cdot \left( \frac{\theta(u)}{\theta(u/q^s)} \right)|_{u=1} \otimes \frac{1}{m}
\\
&= \theta(1/a)^m \cdot (-1)^s q^{s(s-1)/2}\otimes \frac{1}{m}
=q^{s(s-1)/2} \otimes \frac{1}{m},
\end{align*}
where we used \eqref{eq:theta-q} again.
When $[b]=[a]$ we proceed similarly:
\begin{align*}
\lc_{[a]}(f, m)
&= \left( \frac{f(u)}{\theta(u/a)^m} \right)|_{u=a} \otimes \frac{1}{m}
= \frac{1}{\theta(a)^{m-1} \theta(a/q^s)}  \otimes \frac{1}{m}
\\
&= \theta(a)^{-m} \cdot (-1)^s q^{s(s-1)/2} a^{-s}\otimes \frac{1}{m}
=q^{s(s-1)/2} a^{-s} \otimes \frac{1}{m}.
\end{align*}
Finally, it is easier when $[b] \not= [1], [a]$:
\begin{align*}
\lc_{[b]}(f, m)
&= f(b) 
\otimes \frac{1}{m}
= \left( \frac{\theta(b/a)}{\theta(b)} \right)^m \frac{\theta(b)}{\theta(b/q^s)}
\otimes \frac{1}{m}
=q^{s(s-1)/2} b^{-s} \otimes \frac{1}{m}.
\end{align*}
We have shown \eqref{eq:llc-theta},
from which (1) readily follows.

(2)
Since $s_m([\zeta])=0$
for any $m>0$ and $\zeta \in \mu_m(k)$,
the right hand side $[\mu(k)]$ of (2)
is contained in $E_q(k)_\Tor^\is$ by (1).
Hence $\langle \cdot, \cdot \rangle$ factors through
a symmetric biadditive pairing on $T:=E_q(k)_\Tor^\is/[\mu(k)]$.
From \eqref{eq:ex-TateCvTor}
we find that $T$ is a finite cyclic group 
and $m_T:=|T|$ divides $n$.
Take $a \in k^\times$ such that
the class of $[a] \in E_q(k)_\Tor$ generates $T$,
and let $m$ be the order of $[a]$ in $E_q(k)_\Tor$
(which can be strictly larger than $m_T$).
Note that we have
$m/m_T=\min\{ \nu>0 \mid q^\nu \in (k^\times)^m \}$
and $a^m=\zeta q^{cm/m_T}$ for some $\zeta \in \mu_m(k)$ and $c \in \Z$
with $c$ coprime to $m_T$.
Now (1) shows that $\langle [a], [a] \rangle = a \otimes (-c^2/m_T) \in k^\times \otimes \Q/\Z$,
which has order  precisely $m_T$  as seen from 
the exact sequence
\[ \mu(k) \overset{m}{\to} \mu(k) \to k^\times/(k^\times)^m \to k^\times \otimes \Q/\Z.
\]
Thus $[a]^\nu$ belongs to $E_q(k)_\Tor^\is$ if and only if 
$\nu$ is a multiple of $m_T$.
We are done.
\end{proof}

%\begin{corollary}
%Let $\ell, p$ be prime numbers, and $E$ an elliptic curve over $\Q$.
%Suppose that $E(\Q)_\Tor^\is$ contains an element $P$ of order $\ell$,
%and that $E$ has split multiplicative reduction at $p$ 
%of Kodaira type $I_n$ for some $n \in \Z_{>0}$.
%If $p-1$ is prime to $\ell$, 
%then $n$ is divisible by $\ell^2$.
%\end{corollary}
%\begin{proof}
%The assumption implies that
%the base change of $E$ to $\Q_p$ is 
%isomorphic to the Tate curve $E_q=\G_m/q^\Z$
%for some $q \in \Q_p$ such that $\ord_p(q)=n$.
%We apply the above discussion with $m=\ell$.
%Let us write the image of $P$ in $E_q(\Q_p)$ as $[a]$ with $a \in \Q_p$.
%Since the assumption implies $\mu_\ell(\F_p)=\{ 1 \}$,
%we find from  \eqref{eq:ex-TateCv} that
%$E_q(\Q_p)[\ell]$ is a cyclic group of order $\ell$ generated by $[a]$.
%In particular, $s:= s([a])$ does not vanish in $\Z/\ell\Z$.
%We have $\langle [a], [a] \rangle = 0$ since $P \in E(\Q)_\Tor^\is$.
%On the other hand,
%Proposition \ref{prop:TateCv} shows that
%$\langle [a], [a] \rangle = q^{-s^2} \otimes (1/\ell^2)$.
%By looking at its image under the map
%$\ord_p \otimes \id_{\Q/\Z} : \Q_p^\times \otimes \Q/\Z \to \Q/\Z$,
%we conclude $-s^2 n/\ell^2 = 0$ in $\Q/\Z$,
%that is, $\ell^2$ divides $n$.
%\end{proof}

\begin{corollary}\label{cor:kodairatype}
Let $E$ be an elliptic curve over a number field $k$,
and $P \in E(k)_\Tor^\is$ an element of the intrinsic subgroup of order $m>0$.
Let $v$ be a finite place of $k$
such that the completion $k_v$ of $k$ at $v$ satisfies $\mu_m(k_v)= \{ 1 \}$.
If $E$ has split multiplicative reduction of Kodaira type $I_n$ at $v$,
then $n$ is divisible by $m^2$.
\end{corollary}
\begin{proof}
The assumption implies that
the base change of $E$ to $k_v$ is 
isomorphic to the Tate curve $E_q=\G_m/q^\Z$
for some $q \in k_v$ such that $v(q)=n$
(\cite[Chapter V, Theorem 5.3]{sil})
Let us write the image of $P$ in $E_q(k_v)$ as $[a]$ with $a \in k_v$.
Since $\mu_m(k_v)= \{ 1 \}$, 
we find from  \eqref{eq:ex-TateCv} that
$E_q(k_v)[m]$ is a cyclic group of order $m$ generated by $[a]$.
In particular, $s:= s_m([a])$ is invertible in $\Z/m\Z$.
We have $\langle [a], [a] \rangle = 0$ since $P \in E(k)_\Tor^\is$.
On the other hand,
Proposition \ref{prop:TateCv} (1) shows
$\langle [a], [a] \rangle = q^{-s^2} \otimes (1/m^2)$.
By looking at its image under the map
$v \otimes \id_{\Q/\Z} : k_v^\times \otimes \Q/\Z \to \Q/\Z$,
we conclude that $-s^2 n/m^2 = 0$ holds in $\Q/\Z$,
that is, $m^2$ divides $n$.
\end{proof}

\begin{example}
Let $E$ an elliptic curve over $\Q$
such that $E(\Q)_\Tor^\is$ contains an element of order $m>0$.
Let $p$ be an odd prime such that $p-1$ is prime to $m$.
If
$E$ has split multiplicative reduction of Kodaira type $I_n$ at $p$,
then $n$ is divisible by $m^2$.
\end{example}

%\begin{remark}
%The same proof shows that $m^2|n$ holds 
%also under the assumption that
%$E$ has non-split multiplicative reduction at $p$ 
%and $p^2-1$ is prime to $m$.
%\end{remark}

\bibliographystyle{plain}
%\bibliography{intrinsic-bib}

%\begin{thebibliography}{99}

%\bibitem{B}
%Spencer Bloch. 
%Algebraic cycles and higher $K$-theory. 
%Adv. in Math. 61 No 3 (1986) 267--304. 
%

%\bibitem{L}
%Marc Levine.
%Techniques of localization in the theory of algebraic cycles, 
%J. Alg. Geom. 10. (2001) 299--363. 

%\bibitem{M}
%Mar Curc\'o Iranzo.
%Rational torsion of generalised modular Jacobians of odd level.
%\url{https://arxiv.org/abs/2112.03741}

%\bibitem{Mazur}
%Barry Mazur.
%Modular curves and the {E}isenstein ideal.
%Inst. Hautes \'{E}tudes Sci. Publ. Math., (47):33--186 (1978), 1977.

%\bibitem{Ohta}
%Masami Ohta.
%Eisenstein ideals and the rational torsion subgroups of modular
%  {J}acobian varieties {II}.
%Tokyo J. Math., 37(2):273--318, 2014.

%\bibitem{sil}
%Joseph H. Silverman.
%Advanced Topics in the Arithmetic of Elliptic Curves. 
%Springer-Verlag, 1994.

%\bibitem{Totaro}
%Burt Totaro. 
%Milnor $K$-theory is the simplest part of algebraic K-theory. 
%$K$-Theory 6 (1992), no. 2, 177--189.

%\bibitem{WY}
%Fu-Tsun Wei, Takao Yamazaki. 
%Rational torsion of generalized Jacobians of modular and Drinfeld
%modular curves. 
%Forum Mathematicum, vol. 31, no. 3, 2019, pp. 647--659.

%\bibitem{YY}
%Takao Yamazaki, Yifan Yang.
%Rational torsion on the generalized Jacobian 
%of a modular curve with cuspidal modulus. 
%Doc. Math., 21 (2016), 1669--1690. 

%\end{thebibliography}
\end{document}